\documentclass[reqno,11pt]{amsart}
\usepackage{amsfonts,amsmath,amsthm, url}
\usepackage{amssymb}  
\usepackage{mathrsfs}
\usepackage{graphics}
\usepackage{graphicx}
\usepackage{color}
\usepackage{stmaryrd} 
   \newtheorem{thm}{Theorem}
   \newtheorem{prop}{Proposition}
   \newtheorem{lem}{Lemma}[section]
   \newtheorem{cor}[lem]{Corollary}
   \newtheorem{rem}{Remark}[section]

\newcommand{\N}{\mathbb{N}}

\newcommand{\R}{\mathbb{R}}
\newcommand{\C}{\mathbb{C}}
\newcommand{\E}{\mathbb{E}}
\newcommand{\cH}{\mathcal{H}}
\newcommand{\prob}{\mathbb{P}}

\newcommand{\re}{\mathop{\mathrm{Re}}\nolimits}
\newcommand{\im}{\mathop{\mathrm{Im}}\nolimits}
\newcommand{\eps}{\varepsilon}

\newcommand{\ds}{\displaystyle}

\newcommand{\vertiii}[1]{{\left\vert\kern-0.25ex\left\vert\kern-0.25ex\left\vert #1 
    \right\vert\kern-0.25ex\right\vert\kern-0.25ex\right\vert}}

\makeatletter
  
  \@addtoreset{equation}{section}
\makeatother
\begin{document}

\title[Two dimensional Gross Pitaevskii equation with space-time white noise]
{Two dimensional Gross-Pitaevskii equation with space-time white noise}
\author[Anne DE BOUARD]{Anne DE BOUARD$^{\scriptsize 1}$}
\author[Arnaud DEBUSSCHE]{Arnaud DEBUSSCHE$^{\scriptsize 2,3}$}
\author[Reika FUKUIZUMI]{Reika FUKUIZUMI$^{\scriptsize 4}$}



\maketitle

\begin{center} \small
$^1$ CMAP, CNRS, Ecole polytechnique, I.P. Paris\\
91128 Palaiseau, France; \\
\email{anne.debouard@polytechnique.edu}
\end{center}

\begin{center} \small
$^2$ Univ Rennes, CNRS, IRMAR - UMR 6625, F-35000 Rennes, France;\\
\end{center}

\begin{center} \small
$^3$ Institut Universitaire de France (IUF);\\
\email{arnaud.debussche@ens-rennes.fr}
\end{center}

\begin{center} \small
$^4$ Research Center for Pure and Applied Mathematics, \\
Graduate School of Information Sciences, Tohoku University,\\
Sendai 980-8579, Japan; \\
\email{fukuizumi@math.is.tohoku.ac.jp}
\end{center}

\vskip 0.1 in
\noindent
{\small 
{\bf Abstract}. In this paper we consider the two-dimensional stochastic Gross-Pitaevskii equation, 
which is a model to describe Bose-Einstein condensation at positive temperature. The equation 
is a complex Ginzburg Landau equation with a harmonic potential and 
an additive space-time white noise. We study the well-posedness of the model using an inhomogeneous Wick 
renormalization due to the potential, and prove the existence of an invariant measure and of stationary martingale solutions.}
\vskip 0.1 in

\section{Introduction} 

The paper is concerned with the mathematical analysis of the two-dimensional Gross-Pitaevskii equation, 
which is a model for Bose-Einstein condensates in the presence of stochastic effects, e.g. temperature
effects arising around the critical temperature of condensation.
Interactions of the condensate with the ``thermal cloud'' formed by non-condensed atoms need to be taken into account
in this situation.
Those interactions should preserve the principles of the fluctuation-dissipation theorem, which ensures
formally the relaxation of the system to the expected physical equilibrium (see \cite{bbdbg, ds,gd}), leading to the so-called
Projected Gross-Pitaevskii equation.

Neglecting the projection operator on the lowest energy modes, and setting the chemical potential to zero, the equation 
for the macroscopic wave function $\psi$ may be written in its simplest dimensionless form :
\begin{equation} \label{eq:SGPE_Phys}
\partial_t \psi=(i+\gamma) \big(\Delta \psi -V(x) \psi -g|\psi|^2 \psi \big) + \dot W_\gamma(t,x),
\end{equation}
where $V(x)$ is a confining, generally harmonic, potential, and $\dot W_\gamma$ is a two-dimensional space-time white noise
that is a Gaussian field with  delta correlations in time and space :
$$\langle \dot W_\gamma^*(s,y), \dot W_\gamma(t,x)\rangle= 2\gamma \delta_{t-s} \delta_{x-y}.$$
The constant $g$ is a positive physical constant.
When $\gamma=0$, one recovers the standard Gross-Pitaevskii equation for the wave function $\psi$,
and in this case the hamiltonian
$$\mathcal{H}(\psi) = \frac12 \int_{\R^2}  |\nabla \psi(x)|^2 dx + \frac12\int_{\R^2}  V(x)|\psi(x) |^2 dx +\frac{g}{4} \int_{\R^2} |\psi(x)|^4 dx$$
is conserved. 
It follows that for $\gamma>0$, a formal Gibbs measure for \eqref{eq:SGPE_Phys} is given by 
$$ \rho(d\psi) =\Gamma \exp\left[-\mathcal{H}(\psi)\right]d\psi,$$
for some normalizing constant $\Gamma$. 

In this article we consider the harmonic potential $V(x)=|x|^2$.
Equation \eqref{eq:SGPE_Phys} has been studied in space dimension one in \cite{dbdf}, and the existence of
global solutions for all initial data was proved. The convergence to equilibrium was also obtained in \cite{dbdf}, thanks to
a Poincar\'{e} inequality, and to the properties of the invariant measure previously proved in \cite{btt}, where its support was
in particular shown to contain $L^p(\R)$ for any $p>2$.

The aim of the present paper is to extend part of those results to the two-dimensional case. Note that nothing was known in this
case about the Gibbs measure, which has to be built. A first remark to be done is that the Gaussian measure generated by
the linear equation is only supported in $\mathcal{W}^{-s,q}$, with $s>0$, $q \ge 2$, and $sq>2$ (here, $\mathcal{W}^{-s,q}$
is a Sobolev space based on the operator $-\Delta +V$, see Section 2 below). Hence, as is the case for the stochastic quantization equations,
the use of renormalization is necessary in order
to give a meaning to the solutions of \eqref{eq:SGPE_Phys} in the support of the latter Gaussian measure.

Renormalization procedures, using Wick products, have been by now widely used in the context of stochastic partial differential
equations (see e.g. for the case of dimension $2$ considered here \cite{dpd, DPD, dpt, mw, tw} and references therein), 
in particular for parabolic equations based on gradient flows. The complex Ginzburg-Landau equation
driven by space-time white noise, i.e. \eqref{eq:SGPE_Phys} without the confining potential, posed on the three-dimensional torus,
was studied in \cite{h} and for the two-dimensional torus in \cite{m, trenberth}. The main difference in our case is the presence of the harmonic potential $V$. 
It is thus natural to use
functional spaces based on the operator $-H=-\Delta +V$, rather than on standard Sobolev or Besov spaces;
we chose to work on Sobolev spaces based on $-H$ since it is enough for our analysis, 
but we need to prove all the necessary product rules in these spaces. 

Several difficulties arise when trying to adapt the previous methods to the present two-dimensional case. First, the diverging constant
in the definition of the Wick product is no more a constant, but rather a function of the space variable $x$; This is already the 
case for SPDEs on manifolds for instance and does not imply many difficulties. However,  up to our knowledge, 
Wick products corresponding to the Gaussian measure associated to the operator $-H$ considered here have never been constructed. 
An essential tool in the definition of the Wick products is the kernel $K(x,y)$ of the operator $(-H)^{-1}$, and in particular its integrability properties.
It appears (\cite{r}) that $K$ is never in $L^p(\R_x\times \R_y)$, for any $p\ge 1$, but we only have $K\in L^r(\R_x;L^p(\R_y))$
for $r>p\ge 2$ (see Proposition \ref{prop:Kernel} below).

Using these properties of the kernel $K$, we construct the Wick products with respect to the Gaussian measure with covariance $(-H)^{-1}$ and 
use the method of \cite{DPD} to construct local solutions. Then using ideas from \cite{mw}, we are able to prove that the solutions are global 
when $\gamma$ is sufficiently large. 
Moreover, we prove that \eqref{eq:SGPE_Phys} has an invariant measure which is the limit of Gibbs measures 
corresponding to finite dimensional approximations of this equation. This can be seen as a construction of the infinite dimensional Gibbs measure $\rho$. 

Since the kernel $K$ is not in $L^4(\R_x\times \R_y)$, we strongly believe that, contrary to the space dimension one, 
this Gibbs measure is singular with respect to the equilibrium Gaussian 
measure of the linear equation.

It is expected that \eqref{eq:SGPE_Phys} has a unique invariant measure. Strong Feller property of the associated transition
semigroup can be proved using similar arguments as in \cite{tw} or \cite{dpd2}. Unfortunately, irreducibility seems to be 
much more difficult. This question will be the object of a future work.  

Another problem we encounter is that we are not able to prove global existence for any $\gamma$. We cannot use the same argument as 
in \cite{dbdf}. This would need a better understanding of the Gibbs measure $\rho$. Instead, we construct martingale stationary solutions when 
$\gamma$ is small.

\section{Preliminaries and main results.} 
Writing equation (\ref{eq:SGPE_Phys}) in a more mathematical form, 
we will consider in what follows the infinite dimensional, stochastic complex Ginzburg-Landau equation, with a harmonic potential:
\begin{equation} \label{eq:SGPE_Math} 
dX=(\gamma_1 + i\gamma_2)(HX -|X|^{2} X)dt +\sqrt{2\gamma_1}dW, \quad t>0, \quad x \in \R^2,
\end{equation}
where $H=\Delta-|x|^2,$ $x\in \R^2$. We consider a more general equation with parameters 
$\gamma_1 >0$, and $\gamma_2 \in \R$, in order  to clarify 
the effects of the dissipation induced by $\gamma_1$. Let $\{h_k\}_{k \in \N^2}$ be the orthonormal basis 
of $L^2(\R^2, \R)$, consisting of eigenfunctions of $-H$ 
with corresponding eigenvalues $\{\lambda_k^2\}_{k\in \N^2}$, i.e. $-H h_k=\lambda_k^2 h_k$. 
It is known that $\lambda_k^2=2|k|+2$ with $k=(k_1,k_2) \in \N^2$, and the functions $h_k(x)$ are the Hermite functions. 
The unknown function $X$ is a complex valued random field on a probability space $(\Omega, \mathcal{F}, \prob)$ endowed 
with a standard filtration $(\mathcal{F}_t)_{t\ge 0}$. 

We take $\{h_k , i h_k \}_{k \ge 0}$ as a complete orthonormal
system in $L^2(\R^2, \C)$, and we may write the cylindrical Wiener process as  
\begin{equation} \label{WienerProcess}
W (t, x)= \sum_{k\in \N^2} (\beta_{k,R} (t) + i\beta_{k,I}(t)) h_k (x).
\end{equation}
Here, $(\beta_{k,R} (t))_{t \ge 0}$ and $(\beta_{k,I} (t))_{t \ge 0}$ are sequences of independent real-valued Brownian motions, 
on the stochastic basis $(\Omega, \mathcal{F}, \prob, (\mathcal{F}_t)_{t\ge 0}).$ In all what follows, the notation $\E$ stands for
the expectation with respect to $\prob$.
\vspace{3mm}

For $1 \le p \le +\infty,$ and $s\in \R,$ we define the Sobolev space associated to the operator $H$:
\begin{equation*}
\mathcal{W}^{s,p}(\R^2)=
\{v \in \mathcal{S}'(\R^2), \; |v|_{\mathcal{W}^{s,p}(\R^2)}:=|(-H)^{s/2} v|_{L^p(\R^2)} <+\infty \},
\end{equation*}
where $\mathcal{S}$ and $\mathcal{S}'$ denote the Schwartz space and its dual space, respectively.
If $I$ is an interval of $\R$, $E$ is a Banach space,
and $1\le r\le \infty$, then $L^r(I,E)$ is the space of
strongly Lebesgue measurable functions $v$ from $I$ into $E$
such that the function $t\to |v(t)|_E$ is in $L^r(I)$.
We define similarly the spaces
$C(I,E)$, $C^{\alpha}(I, E)$ or $L^r(\Omega, E)$. 
For a complex Hilbert space $E$, 
the inner product will be understood as taking the real part, i.e., for $u=u^R+iu^I \in E$ and 
$v=v^R+iv^I \in E$, then we set $(u, v)_{E}:=(u^R, v^R)_{E} + (u^I, v^I)_{E}$    
that is we use the identification $\C \simeq \R^2$. 
The bracket notation is used for the meaning $\langle f \rangle := (1+|f|^2)^{1/2}.$ 
We denote by $E_N^{\C}$ the complex vector space spanned by the Hermite functions, $E_N^\C=\mathrm{span}\{h_0, h_1,..., h_N\}$. 
If we write $E_N^{\R}$ this means the real vector space instead. 
\vspace{3mm}

In the course of the proofs, we will frequently use an approximation by finite dimensional objects.
In order to clarify the convergence properties of a function series of the form 
\linebreak
$u=\sum_{k\in \N^2} c_k h_k,$
we define, for any $N \in \N$ fixed, for any $p \in [1,\infty]$, and $s\in \R$, 
a smooth projection operator $S_N: L^2(\R^2,\C)  \to E_N^{\C}$ by 
\begin{equation}
\label{def:SN}
S_N \Big[\sum_{k \in \N^2} c_k h_k \Big]:= \sum_{k \in \N^2} \chi \Big[\frac{\lambda_k^2}{\lambda_N^2}\Big] c_k h_k 
=\chi\Big[\frac{-H}{\lambda_N^2}\Big] \Big[\sum_{k \in \N^2} c_k h_k\Big], 
\end{equation}
where $\chi \ge 0$ is a cut-off function such that $\chi \in C_0^{\infty}(-1,1)$, $\chi=1$ on $[-\frac{1}{2}, \frac{1}{2}].$ 
Note that here and in what follows, we denote by $\lambda_N$ the value $\lambda_{(N,0)}$, for simplicity.
The operator $S_N$, which is self-adjoint and commutes with $H$, may easily be extended by duality to any Sobolev space $\mathcal{W}^{s,2}(\R^2; \C)$,
with $s\in \R$, and thus by Sobolev embeddings, to any space $\mathcal{W}^{s,p}(\R^2; \C)$, with $p\ge 1$.
We will use of the following lemma, whose proof is similar to Theorem 1.1 of \cite{jn}.
\begin{lem} \label{prop:specmap} For any $\psi \in \mathcal{S}(\R)$ and any $p \in [1,\infty]$, there exists
a constant $C = C(\psi) > 0$ such that
$$|\psi(-\theta H)|_{\mathcal{L}(L^p, L^p)} \le C,$$
for any $\theta \in (0,1)$. 
\end{lem}
This lemma implies that  
$S_N$ is a bounded operator from $L^p$ to $L^p$, uniformly in $N$,  
for any $p \in [1,\infty]$. Note that the usual spectral projector, 
$$\Pi_{N}\Big[\sum_{k \in \N^2} c_k h_k \Big]:= \sum_{k \in \N^2, |k|\le N} c_k h_k,$$ 
does not satisfy this property. 
\vspace{3mm} 

As was pointed out in the introduction, due to the space-time white noise, the solution of (\ref{eq:SGPE_Math}) is expected to 
have negative space regularity, and thus the nonlinear term $-|X|^2 X$ is ill-defined.  
In order to make sense of this term, we use a renormalization procedure based on Wick products. 
This amounts to ``subtract an infinite constant" from the nonlinear term
in \eqref{eq:SGPE_Math}. 
More precisely, writing the solution $X=u+Z_{\infty}^{\gamma_1, \gamma_2}$ with  
\begin{equation} \label{eq:Zinfini}
Z_{\infty}^{\gamma_1, \gamma_2}(t)=\sqrt{2\gamma_1}\int_{-\infty}^t e^{(t-\tau) (\gamma_1+i\gamma_2) H} dW(\tau),  
\end{equation}
which is the stationary solution for the linear stochastic equation
\begin{equation} \label{eq:Z}
dZ=(\gamma_1+i\gamma_2)HZdt +\sqrt{2\gamma_1} dW,  
\end{equation}
we find out the following random partial differential equation for $u$:
\begin{equation} \label{eq:u}
\partial_t u = (\gamma_1+i\gamma_2) (Hu-|u+Z_{\infty}^{\gamma_1, \gamma_2}|^2(u+Z_{\infty}^{\gamma_1, \gamma_2})), \quad u(0)=u_0:=X(0)-Z_{\infty}^{\gamma_1, \gamma_2}(0). 
\end{equation}
We are therefore required to solve this random partial differential equation. However, using standard arguments, it is not difficult to find that 
the best regularity we may expect for $Z_\infty^{\gamma_1, \gamma_2}$ is almost surely:
$Z_{\infty}^{\gamma_1, \gamma_2} \in \mathcal{W}^{-s,q}(\R^2)$
for $s>0$, $q \ge 2,$ $sq>2$, see Lemma \ref{lem:Kolmogorov} below.
Developing 
\begin{eqnarray}
\label{develop}
|u+Z_{\infty}^{\gamma_1, \gamma_2}|^2(u+Z_{\infty}^{\gamma_1, \gamma_2})&=&|u|^2 u +2|u|^2 Z_{\infty}^{\gamma_1, \gamma_2}
+ \bar{u} ({Z_{\infty}^{\gamma_1, \gamma_2}})^2 +u^2 \overline{Z_{\infty}^{\gamma_1, \gamma_2}} \nonumber \\[0.2cm]
&  & +\, 2u|Z_{\infty}^{\gamma_1, \gamma_2}|^2 
+|Z_{\infty}^{\gamma_1, \gamma_2}|^2 {Z_{\infty}^{\gamma_1, \gamma_2}},
\end{eqnarray}
we are led to multiply functions having both negative Sobolev regularity, which cannot be defined in the usual distribution sense.  
\vspace{3mm}

We need some preliminaries before we introduce the renormalization procedure.  
Let us recall a few facts about Hermite polynomials $H_n(x)$, $n \in \N$,
These are defined through the generating functions
$$ e^{-\frac{t^2}{2}+t x} =\sum_{n=0}^{\infty} \frac{t^n}{\sqrt{n !}} H_n(x), \quad x, t \in \R,$$ 
where 
\begin{equation} \label{eq:HermitePoly}
H_n(x)=\frac{(-1)^n}{\sqrt{n!}} e^{\frac{x^2}{2}} \frac{d^n}{dx^n}(e^{-\frac{x^2}{2}}), \quad n\ge 1
\end{equation}
and $H_0(x)=1$. 
\vspace{3mm}
The Wick products of $Z^{\gamma_1, \gamma_2}_{\infty}$ are defined as follows. We write $Z_{R,\infty}=\mathrm{Re} Z^{\gamma_1, \gamma_2}_{\infty}$ and 
$Z_{I, \infty}=\mathrm{Im} Z^{\gamma_1, \gamma_2}_{\infty}$. 
For any $k,l\in \N$, we define 
\begin{eqnarray*}
:(Z_{R,\infty})^{k} (Z_{I,\infty})^{l}:~ &:=& \lim_{N\to\infty} :(S_N Z_{R,\infty})^{k}: :(S_N Z_{I,\infty})^{l}:, \quad \mbox{in} \quad L^q(\Omega, \mathcal{W}^{-s,q}(\R^2)),
\end{eqnarray*}
where $s>0, q\ge 4$ and $sq>8$. In the right hand side,  the notation $:(S_N z)^n:(x)$ for $n \in \N$, $N\in \N$, $x\in \R^2$, and for a real-valued 
centered Gaussian white noise $z$, means 
$$:(S_N z)^n :(x) = \rho_N(x)^n \sqrt{n !} H_n\left[\frac{1}{\rho_N(x)} S_N z(x)\right], \quad x \in \R^2$$ 
with 
$$\rho_N(x)=\left[\sum_{k \in \N^2} \chi^2\left(\frac{\lambda_k^2}{\lambda_N^2}\right) \frac{1}{\lambda_k^2} (h_k(x))^2\right]^{\frac12}.$$
More details about the Wick products will be given in Section 3.1, where the above convergence will be proved, and 
we will see that the Wick product is indeed well-defined in $\mathcal{W}^{-s,q}$, as soon as $s>0$, $q\ge4$ and $qs>8$ (see Proposition 
\ref{prop:ReImCauchy}).
\vspace{3mm}


We thus consider the following renormalized equation in the space $\mathcal{W}^{-s,q}(\R^2)$ (in the weak sense): 
\begin{eqnarray} 
\label{eq:SGL}
dX &=&(\gamma_1+i\gamma_2)(HX-:|X|^2X:)dt +\sqrt{2\gamma_1}dW, \quad t>0, \quad x\in \R^2, \\ \nonumber 
X(0) &=&X_0, 
\end{eqnarray} 
in the sense that we solve the shifted equation (\ref{eq:u}) with $|u+Z_{\infty}^{\gamma_1, \gamma_2}|^2(u+Z_{\infty}^{\gamma_1, \gamma_2})$ 
replaced by 
$:|u+Z_{\infty}^{\gamma_1, \gamma_2}|^2(u+Z_{\infty}^{\gamma_1, \gamma_2}):$ defined by
replacing in \eqref{develop} all the terms involving $Z_{\infty}^{\gamma_1, \gamma_2}$ by the corresponding wick products (see
\eqref{def:F} for a more precise definition).

Hence we consider
\begin{equation} \label{eq:u_wick}
\partial_t u = (\gamma_1+i\gamma_2) \left[Hu-:|u+Z_{\infty}^{\gamma_1, \gamma_2}|^2(u+Z_{\infty}^{\gamma_1, \gamma_2}):\right], \end{equation}
supplemented with the initial condition
$$
 u(0)=u_0=X_0-Z_{\infty}^{\gamma_1, \gamma_2}(0). 
$$

\vspace{3mm}

We will first prove the following local well-posedness result of equation \eqref{eq:u_wick}.
\begin{thm} \label{thm:localexistence} Fix any $T>0$. Let $\gamma_1>0$, $\gamma_2 \in \R$ and  $q>p>3r$, $r>6$. Assume 
$0<s<\beta<2/p$, $qs>8$, 
$ \beta-s>\frac{2}{p}-\beta$
and 
$s+2 \big(\frac{2}{p} -\beta\big) < 2\big(1-\frac{1}{q}\big).$
Let $u_0 \in \mathcal{W}^{-s,q}(\R^2)$. 
Then there exists a random stopping time $T_0^*(\omega)>0$, 
which depends on $u_0$ and $(:(\mathrm{Re}Z_{\infty}^{\gamma_1, \gamma_2})^k 
(\mathrm{Im}Z_{\infty}^{\gamma_1, \gamma_2})^l:)_{0\le k+l\le 3}$,
and a unique solution $u$ of (\ref{eq:u_wick}) such that 
$u \in C([0,T_0^*), W^{-s,q}(\R^2))\cap L^r(0,T_0^*, W^{\beta,p}(\R^2))$ a.s. We have moreover 
almost surely $T^*_0=T$ or $\lim_{t \uparrow T^*_0} |u(t)|_{\mathcal{W}^{-s,q}}=+\infty.$
\end{thm}
\vspace{3mm}

When the dissipation coefficient $\gamma_1$ is sufficiently large, an energy estimate allows to get a bound on
the $L^q$  norm of the solution, and to deduce a global existence result, as is stated in the next Propositions and
Theorem. 
This method of globalization has been widely used for the complex Ginzburg-Landau equation (see \cite{bs,dgl}) and has been 
adapted in the renormalized case (\cite{h,mw}).
\vspace{3mm}

\begin{prop} \label{prop:LpLWP}
Let $\gamma_1>0$, $\gamma_2 \in \R$ and $q>p>3r$, $r>6$. Assume 
$0<s<\beta<2/p$ 
satisfy the assumptions of Theorem \ref{thm:localexistence} and that we have in addition
$ s\le \frac{2}{p} -\beta <\frac{1}{12}$, with $2(\frac{2}{p}-\beta)<\beta $ and $3(\frac{2}{p}-\beta)<2(1-\frac{1}{q})$.
Let $u_0 \in L^q(\R^2)$. 
Then the solution $u$ of \eqref{eq:u_wick} given by Theorem \ref{thm:localexistence} satisfies : $u\in C([0,T_0^*), L^q(\R^2))$.
\end{prop}
\vspace{3mm}

In the next Proposition we give the $L^q$ a priori bound.

\begin{prop} \label{thm:Lpbound}[$L^q$ a priori estimate]. Let $\gamma_1>0$ and $q>p>3r$, $r>6$. Assume 
$0<s<\beta<2/p$
satisfy the assumptions of  Theorem \ref{thm:localexistence} and we have in addition 
$ s\le \frac{2}{p} -\beta <\frac{1}{12}$, with $2(\frac{2}{p}-\beta)<\beta $ and $3(\frac{2}{p}-\beta)<2(1-\frac{1}{q})$. 
Moreover assume $\gamma_2=0$, or $q< 2+2(\kappa^2+\kappa \sqrt{1+\kappa^2})$ with $\kappa=\gamma_1/\gamma_2$ if $\gamma_2 \ne 0$.
Let $u_0 \in L^q(\R^2)$, and let $u$ be the unique solution 
constructed in Theorem \ref{thm:localexistence}.  
Then, there exists a constant $C>0$ depending on $\gamma_1, \gamma_2, q$ and
$(:(\mathrm{Re}Z_{\infty}^{\gamma_1, \gamma_2})^k 
(\mathrm{Im}Z_{\infty}^{\gamma_1, \gamma_2})^l:)_{0\le k+l\le 3}$, such that  for any $t$ with $0<t<T_0^*$,
$$ |u(t)|_{L^q}^q \le e^{-\frac{\gamma_1 t \delta}{4}}|u_0|_{L^q}^q +C,$$
where $T_0^*$ is the maximal existence time given in Theorem \ref{thm:localexistence}. The coefficient $\delta$ is given by $ \delta = 1$ if $\gamma_2=0$, 
and $\delta = 1-\frac{q-2}{2(\kappa^2+\kappa\sqrt{1+\kappa^2})}$ if $\gamma_2\not =0$.
\end{prop}
\vspace{3mm}

Gathering the previous results and using the smoothing properties of the heat semi-group, 
we finally obtain the global existence of solutions in the case of large dissipation.
\vspace{3mm}

\begin{thm} \label{thm:global} Let $\gamma_1>0$ and $q>p>3r$, $r>6$. Assume 
$0<s<\beta<2/p$ 
satisfy the assumptions of  Theorem \ref{thm:localexistence} and we have in addition 
$ s\le \frac{2}{p} -\beta <\frac{1}{12}$, with $2(\frac{2}{p}-\beta)<\beta $ and $3(\frac{2}{p}-\beta)<2(1-\frac{1}{q})$. 
Moreover assume $\gamma_2=0$, or $q< 2+2(\kappa^2+\kappa \sqrt{1+\kappa^2})$ with $\kappa=\frac{\gamma_1}{\gamma_2}$ if $\gamma_2 \ne 0$.
Let $u_0 \in \mathcal{W}^{-s,q}(\R^2)$.
Then there exists a unique global solution $u$ of (\ref{eq:u_wick}) in $C([0,T], \mathcal{W}^{-s,q}) \cap L^r(0,T; \mathcal{W}^{\beta,p})$ a.s. for any $T>0$. 
\end{thm}
\vspace{3mm}

Now that we have at hand a global flow, for $\gamma_1$ large enough, the next step is the construction of a Gibbs measure.
The Gibbs measure is formally written as an infinite-dimensional measure of the following form: 
$$ \rho(du)= \Gamma e^{-\mathcal{H}(u)} du,$$
where 
$$\mathcal{H}(u)= \frac{1}{2} \int_{\R^2} |\nabla u (x)|^2 dx + \frac{1}{2}\int_{\R^2} |xu(x)|^2 dx +\frac{1}{4} \int_{\R^2} |u(x)|^4 dx,$$
and $\Gamma$ is a normalizing constant. We will make sense of this infinite dimensional measure as follows.  
Using (\ref{WienerProcess}), it may be easily seen that (\ref{eq:Zinfini}) can be written as 
\begin{equation} \label{Zseries}
Z_{\infty}^{\gamma_1, \gamma_2}(t)=\sum_{k \in \N^2} \frac{\sqrt{2}}{\lambda_k}g_k(\omega,t)h_k(x),
\end{equation}
where $\{g_k (\omega,t)\}_{k \in \N^2}$ is a system of independent, complex-valued random variables with
law $\mathcal{N}_{\C}(0, 1)$.
Thus, the projection onto $E_N^{\C}$ of the stationary solution $Z_{\infty}^{\gamma_1, \gamma_2}(t)$ has the same law as the   
Gaussian measure $\mu_N$ induced by a random series 
$$\varphi_N(\omega,x):= \sum_{k \in \N^2, |k| \le N} \frac{\sqrt{2}}{\lambda_{k}} g_k(\omega) h_k(x)$$
defined on $(\Omega, \mathcal{F}, \prob)$ 
where $\{g_k (\omega)\}_{k \in \N^2}$ is a system of independent, complex-valued random variables with
the law $\mathcal{N}_{\C}(0, 1)$. 
Thanks to the same argument as Lemma 2.1 in \cite{dbdf}, combined with Proposition \ref{prop:Kernel} below, 
the series converges in $L^q(\Omega,\mathcal{W}^{-s,q})$ if $s>0$, $q \ge 2$ and $s q > 2$ (see Lemma \ref{lem:Kolmogorov} below),
and the limit defines the infinite-dimensional Gaussian measure $\mu$ on $\mathcal{W}^{-s,q}$. 
However, although the above Gibbs expression may be formally written as 
$$ \rho(du)= \Gamma e^{-\frac{1}{4} \int_{\R^2} |u(x)|^4 dx} \mu(du),$$
we cannot make sense of it, since $L^4(\R^2)$ is not in the support of $\mu$. For that reason, we should also renormalize
the $L^4$ norm in the Gibbs measure. However, it is not clear how this can be done in a compatible way.
This fact leads us to define the Gibbs measure for (\ref{eq:SGL}) as a limit of 
\begin{equation*}
\tilde{\rho}_N (dy) = \Gamma_N \exp\left\{-\int_{\R^2}\left(\frac{1}{4}|S_Ny(x)|^4 -2\rho_N^2(x)|S_Ny(x)|^2+2\rho_N^4(x)\right)dx\right\} \mu_N (dy), \quad y \in E_N^{\C}, 
\end{equation*}
where $\Gamma_N$ is the normalizing constant. 
We can indeed prove the tightness of the family of  measures $(\tilde{\rho}_N)_{N\in \N}$. 

\begin{thm} 
\label{thm:invmeasure}
Let $\gamma_1$, $\gamma_{2}$, $q,s$ be as in Theorem \ref{thm:global}. Then, there exists an invariant measure $\rho$ supported in
$\mathcal{W}^{-s,q}(\R^2)$, for the transition semi-group associated with equation \eqref{eq:SGL}, which is well-defined according
to Theorem \ref{thm:global}. Moreover, $\rho$ is the weak limit of a subsequence of the family $(\tilde \rho_N)_N$ defined above.
\end{thm}

Note that the measure $\tilde{\rho}_N$ does not depend on $\gamma_1$ or $\gamma_2$, but the tightness
$(\tilde{\rho}_N)_{N\in \N}$ is not induced by a $L^q$ bound as in Proposition \ref{thm:Lpbound}, which does not a priori hold
for the finite dimensional approximations of equation \eqref{eq:u_wick}. We thus have to prove an alternative bound (see Proposition \ref{prop:tightness} in Section  4)
and, unfortunately, this latter bound does not provide higher moment bounds on the measures $\tilde \rho_N$, 
 preventing us to obtain global strong solutions in the small dissipation case, as would be expected. Nevertheless, the bound in Proposition \ref{prop:tightness} allows us 
to construct a stationary martingale solution for any dissipation coefficients:  

\begin{thm} \label{thm:sol_martingale} 
Let $\gamma_1>0$ and $\gamma_2 \in \R$, and let $0<s<1$, $q>8$, $sq>8$. Then, there exists a stationary martingale solution $X$ of \eqref{eq:SGL} having trajectories  
in $C(\R_+, \mathcal{W}^{-s,q})$  and such that for any $t\ge 0$, $\mathcal{L}(X(t))= \rho$, the measure constructed in Theorem \ref{thm:invmeasure}.
\end{thm}
 
\vspace{3mm}

The paper is organized as follows. Section 3 consists of three parts. 
In Section 3.1, we begin with the definition of Wick products suitable to our equation, 
and study the regularity of the Wick products. In Section 3.2, we list up the properties of the heat semigroup $(e^{t(\gamma_1+i\gamma_2)H})_{t\ge 0}$. 
Section 3.3 is devoted to the proof of Theorem \ref{thm:localexistence} introducing a polynomial estimate to deal with the nonlinear terms. 
In Section 4.1, we pay attention to the global existence of solutions in the case of large dissipation. 
We derive the $L^q$ a priori bound and prove that the solution exists globally in time by the smoothing properties of the heat semigroup. 
Furthermore, in Section 4.2, we consider a finite-dimensional approximation to (\ref{eq:SGL}), and the associated finite-dimensional Gibbs measure. 
The tightness of this finite-dimensional family of measures is proved and the limit is obtained 
to be an invariant measure for the case of large dissipation concluding Theorem \ref{thm:invmeasure}. 
Finally, using moment bounds and a compactness argument,
we prove in Section 5 the existence of a stationary martingale solution without restriction on the dissipation parameter. 

\section{Local existence}
\subsection{Preliminary results on the regularity of Wick products} 

We begin with introducing some useful tools in the space $\mathcal{W}^{s,p}(\R^2)$.    
\begin{prop}
\begin{itemize}
\item[(1)] \label{prop:dg}  
For any $p\in (1,\infty)$ and $s\ge 0$, there exists $C>0$ such that  
\begin{equation*}
\frac{1}{C}|f|_{\mathcal{W}^{s,p}(\R^2)} \le |\langle D_x \rangle^s f|_{L^p(\R^2)} +|\langle x \rangle^s f|_{L^p(\R^2)} \le C|f|_{\mathcal{W}^{s,p}(\R^2)}.
\end{equation*}
\item[(2)] Let $\alpha\ge 0$. Then the following estimates hold.
$$|fg|_{\mathcal{W}^{\alpha,q}} \le C(|f|_{L^{q_1}} |g|_{\mathcal{W}^{\alpha, \bar{q_1}}} + |f|_{\mathcal{W}^{\alpha, q_2}} |g|_{L^{\bar{q_2}}}),$$
where $1<q<\infty$,  $q_1, q_2 \in (1, \infty]$, $\bar{q_1}, \bar{q_2} \in [1, \infty)$ with 
$\frac{1}{q}=\frac{1}{q_1}+\frac{1}{\bar{q_1}}=\frac{1}{q_2}+\frac{1}{\bar{q_2}}.$
\end{itemize}
\end{prop}

\begin{proof} The norm equivalence (1) was proved in \cite{dg}, and for the interpolation estimate (2), see Proposition 1.1 of \cite{taylor}.
\end{proof}
\vspace{3mm}

Let us now give some details on the construction of the Wick products.
Let $\xi(x,\omega)$ be a mean-zero (real-valued) Gaussian white noise on $\R^2$ defined by 
$$\xi(x,\omega)=\sum_{n \in \N^2} g_n(\omega) h_n(x),$$
where $\{g_n(\omega)\}_n$ is a system of independent, real-valued random variables with $\mathcal{N}(0,1)$ law.
For $f \in L^2(\R^2, \R)$, we define the mapping $W_{(\cdot)}: L^2(\R^2) \to L^2(\Omega)$, $f \mapsto W_f$, 
by  
$$ W_f(\omega)=(f, \xi(\omega))_{L^2}= \sum_{n\in \N^2} (f, h_n)_{L^2} g_n(\omega).$$
Then, $W_f$ is a mean-zero Gaussian random variable with variance $|f|_{L^2(\R^2)}^2$, 
and 
$$\E(W_f W_g)=(f,g)_{L^2}, \quad f,g \in L^2(\R^2, \R).$$  This implies that $W_{(\cdot)}$ is an isometry 
from $L^2(\R^2, \R)$ to $L^2(\Omega)$. Moreover we have the relation 
\begin{equation} \label{eq:Wf}
\E(H_k(W_f) H_l(W_g))=\delta_{k,l} k! (f,g)_{L^2}^k, 
\end{equation}
for any $f,g \in L^2(\R^2)$ with $|f|_{L^2}=|g|_{L^2}=1$, with the use of the Hermite polynomials (\ref{eq:HermitePoly}). 
\vspace{3mm}

Let us define for $n\ge 0$,  the Wiener chaos of order $n$ as
$$\mathcal{H}_n=\overline{\{H_n(W_f): f \in L^2(\R^2), |f|_{L^2}=1\}},$$
where the closure is taken in $L^2(\Omega)$.
It is known that 
$$L^2(\Omega, \mathcal{G}, \prob)=\bigoplus_{n=0}^{\infty} \mathcal{H}_n$$
where $\mathcal{G}$ is the $\sigma$-algebra generated by $\{W_f, f \in L^2(\R^2)\}$ 
(see Theorem 1.1 of \cite{n}). We shall denote by $P_n$ the orthogonal projection of $L^2(\Omega, \mathcal{G}, \prob)$  
on $\mathcal{H}_n$. Then $(W_f)^n \in L^2(\Omega, \mathcal{G}, \prob)$ for any $f\in L^2(\R^2)$ with $|f|_{L^2}=1$, and the following holds.
\begin{lem} \label{lem:WienerChaos} 
Let $f \in L^2(\R^2)$ satisfying $|f|_{L^2}=1$. Then we have, 
$$ P_n ((W_f)^n)= \sqrt{n !} H_n(W_f).$$
\end{lem}
\begin{lem} 
For any $F \in \mathcal{H}_n$, and $p\ge 2$,
\begin{equation} \label{ineq:nelson}
\E(|F|^p) \le (p-1)^{\frac{n}{2} p} \E(|F|^2)^{\frac p2}. 
\end{equation}
\end{lem}
\vspace{3mm}

Let us now define 
\begin{equation}
\label{def:etaN}
\eta_N(x)(\cdot)= \frac{1}{\rho_N(x)} \sum_{k \in \N^2} \chi\left(\frac{\lambda_k^2}{\lambda_N^2}\right) \frac{h_k(x)}{\lambda_k} h_k(\cdot),
\end{equation}
for a fixed $x\in \R^2$, with 
\begin{equation}
\label{def:rhoN}
 \rho^2_N=\rho^2_N(x)=\sum_{k \in \N^2} \chi^2\left(\frac{\lambda_k^2}{\lambda_N^2}\right) \frac{1}{\lambda_k^2} (h_k(x))^2,
 \end{equation}
so that $|\eta_N(x)|_{L^2(\R^2)}=1.$ 
Then if $z(\omega, x)$ is of the form
\begin{equation}
\label{zseries}
z(\omega,x)=\sum_{k\in \N^2} \frac{\sqrt{2}}{\lambda_k} g_k(\omega) h_k(x), \quad \omega \in \Omega, x\in \R^2,
\end{equation}
with a real-valued Gaussian system $\{g_k(\omega)\}_k$ with law $\mathcal{N}(0,1/2)$, we note that  
$S_N z(x)$ makes sense in $L^2(\Omega)$ for any $x\in \R^2$ and 
we may write 
$$ S_N z(x) = \sum_{k\in \N^2} \sqrt{2} g_k (\rho_N(x) \eta_N(x), h_k)_{L^2}= \rho_N(x) {W_{\eta_N(x)}}.$$
We can thus, by Lemma \ref{lem:WienerChaos}, define the Wick product of $(S_N z)^n$ as follows. 
$$:(S_N z(x))^n : =P_n((S_N z(x))^n)= \rho_N(x)^n \sqrt{n !} H_n(W_{\eta_N(x)}), \quad \prob-a.s.$$
For the reader's convenience, we give hereafter the expressions of the first three Wick products.
\begin{eqnarray*}
&& :(S_N z)^1:(x)= S_N z (x), \\
&& :(S_N z)^2:(x)= (S_N z)^2(x) - \rho_N(x)^2, \\
&& :(S_N z)^3:(x)= (S_N z)^3(x) -3 \rho_N(x)^2 S_N z(x). 
\end{eqnarray*}
\vspace{3mm}

Note that the renormalization factor $\rho_N(x)$ depends on $x \in \R^2$, this is different from the torus case (see for ex., \cite{dpd,mw,tw}). 
Since it is known from \cite{kt} that for any $k \in \N^2$,
\begin{equation} \label{est:kt}
|h_k|_{L^p(\R^2)} \lesssim \lambda_k^{-\theta(p)} 
\end{equation}
with 
\begin{equation*}
\theta(p)=\left\{
\begin{array}{ll}
\frac{1}{2}-\frac{1}{p},& \quad \mbox{if} \quad 2 \le p <\frac{10}{3}, \\[0.2cm]
\frac{2}{3p} &  \quad \mbox{if} \quad \frac{10}{3} < p \le +\infty, 
\end{array}
\right.
\end{equation*}
we have for any $p \in [2,+\infty]$,
$$|\rho_N^2|_{L^p(\R^2)} \lesssim N^{1-\theta(2p)} \quad \mathrm{as} \;N \to \infty.$$ 
Note that $\theta(p)$ is at most $\frac15$ so that $\rho_N^2$ may diverge a priori faster than in the torus case.
\vspace{3mm}

Recall that $Z_{\infty}^{\gamma_1, \gamma_2}$ can be written as (\ref{Zseries}). 
In order to consider a renormalization for the nonlinear term in (\ref{eq:SGL}), we decompose $Z_{\infty}^{\gamma_1, \gamma_2}=Z_{R,\infty}+iZ_{I, \infty}$ 
where $Z_{R,\infty}=\mathrm{Re}(Z_{\infty}^{\gamma_1, \gamma_2})$ and $Z_{I,\infty}=\mathrm{Im}(Z_{\infty}^{\gamma_1, \gamma_2})$ are independent,
each of the form \eqref{zseries} and thus equal in law.
Note that all the terms in \eqref{develop} are products of powers of $\mathrm{Re} u, \mathrm{Im} u, Z_{R,\infty}$ and $Z_{I,\infty}$. 
Applying the above construction to $Z_{R,\infty}$ and $Z_{I,\infty}$, we may then define, for any integers $k,l\in \N$, the real valued Gaussian
random variables $:(S_N Z_{R,\infty})^{k}:$ and $:(S_N Z_{I,\infty})^{l}:$, which are still independent.
It is then natural to define, for all integers $k,l \in \N$, the Wick product $:(Z_{R,\infty})^k(Z_{I,\infty})^l:$ as the limit in $N$ of 
the product $:(S_NZ_{R,\infty})^k::(S_NZ_{I,\infty})^l:$. The next proposition indeed shows that this limit is well defined in $L^q(\Omega;W^{-s,q}(\R^2))$
for $q\ge 4$, $s>0$ and $sq>8$. It would not be difficult to prove that this definition coincides with the Wick product that would be obtained
with the use of an $\R^2$-valued white noise measure.

\begin{prop} \label{prop:ReImCauchy}
For any $k,l \in \N$, 
the sequence $\{:(S_N Z_{R,\infty})^{k}: :(S_N Z_{I,\infty})^{l}:\}_{N\in \N}$ is a Cauchy sequence in $L^q(\Omega, \mathcal{W}^{-s,q}(\R^2))$, for $q \ge 4$, $s>0$ with $qs >8$.
\vspace{3mm}

Moreover, defining then, for any $k,l\in \N$,
\begin{eqnarray*}
:(Z_{R,\infty})^{k} (Z_{I,\infty})^{l}:~ &:=& \lim_{N\to\infty} :(S_N Z_{R,\infty})^{k}: :(S_N Z_{I,\infty})^{l}:, \quad \mbox{in} \quad L^q(\Omega, \mathcal{W}^{-s,q}(\R^2)),
\end{eqnarray*}
where $s>0, q\ge 4$ and $sq>8$, there exists a constant $M_{s,q,k,l}$ such that 
\begin{equation} \label{bound:Z}
\E\left[|:(\mathrm{Re}Z_{\infty}^{\gamma_1, \gamma_2})^k (\mathrm{Im}Z_{\infty}^{\gamma_1, \gamma_2})^l:|_{\mathcal{W}^{-s,q}}^q\right] \le M_{s,q,k,l}.
\end{equation}
\end{prop}

\vspace{3mm}

Higher order moments may also be estimated thanks to Nelson formula. The following corollary will be useful later.   

\begin{cor} \label{cor:m-moment} Let $s>0$, $m > q\ge 4$ and $sq>8$.  Then there is a constant $M_{s,q,k,l,m}$ such that
\begin{equation} 
\E\left[|:(\mathrm{Re}Z_{\infty}^{\gamma_1, \gamma_2})^k (\mathrm{Im}Z_{\infty}^{\gamma_1, \gamma_2})^l:|_{\mathcal{W}^{-s,q}}^m\right] \le M_{s,q,k,l,m}.
\end{equation}
\end{cor}

\begin{proof} First, for $m>q$ we apply Minkowski inequality to obtain
\begin{eqnarray*}
& & \E\left[|:(\mathrm{Re}Z_{\infty}^{\gamma_1, \gamma_2})^k (\mathrm{Im}Z_{\infty}^{\gamma_1, \gamma_2})^l:|_{\mathcal{W}^{-s,q}}^m\right]
=  \int_{\Omega} |(-H)^{-s/2} :(\mathrm{Re}Z_{\infty}^{\gamma_1, \gamma_2})^k (\mathrm{Im}Z_{\infty}^{\gamma_1, \gamma_2})^l :|_{L^q}^m d \prob \\
& & \quad \le \Big(\int_{\R^2} \big(\int_{\Omega} |(-H)^{-s/2} :(\mathrm{Re}Z_{\infty}^{\gamma_1, \gamma_2} )^k 
(\mathrm{Im}Z_{\infty}^{\gamma_1, \gamma_2})^l :|^m d\prob\big)^{\frac{q}{m}}dx \Big)^{\frac{m}{q}}. 
\end{eqnarray*}
Here we use Nelson estimate (\ref{ineq:nelson}), and the right hand side is bounded by 
$$(m-1)^{\frac{3}{2}} \big(\int_{\R^2} \E(|(-H)^{-s/2} :(\mathrm{Re}Z_{\infty}^{\gamma_1, \gamma_2})^k (\mathrm{Im}Z_{\infty}^{\gamma_1, \gamma_2})^l:|^2)^{\frac{q}{2}} dx\big)^{\frac{m}{q}}.$$
Again, we apply Minkowski inequality to obtain 
\begin{eqnarray*}
\E\left[|:(\mathrm{Re}Z_{\infty}^{\gamma_1, \gamma_2})^k (\mathrm{Im}Z_{\infty}^{\gamma_1, \gamma_2})^l:|_{\mathcal{W}^{-s,q}}^m\right]
&\le & C_m \E(|:(\mathrm{Re}Z_{\infty}^{\gamma_1, \gamma_2})^k (\mathrm{Im}Z_{\infty}^{\gamma_1, \gamma_2})^l:|_{\mathcal{W}^{-s,q}}^2)^{\frac{m}{2}}\\
 \le C_m \E(|:(\mathrm{Re}Z_{\infty}^{\gamma_1, \gamma_2})^k (\mathrm{Im}Z_{\infty}^{\gamma_1, \gamma_2})^l:|_{\mathcal{W}^{-s,q}}^q)^{\frac{m}{q}}
&\le & C_m M_{s,q,k,l}^{\frac{m}{q}}, 
\end{eqnarray*}
thanks to H\"{o}lder inequality  and (\ref{bound:Z}). 
\end{proof}
\vspace{3mm}

For the proof of Proposition \ref{prop:ReImCauchy}, the key ingredients are estimates on the kernel $K$ of the operator $(-H)^{-1}$, which is defined 
by $ (-H)^{-1} f(x) = \int_{\R^2} K(x,y) f(y) dy$, and may be written as  
\begin{equation}
\label{def:kernel}
K(x, y)=\sum_{k \in \N^2} \frac{h_k(x) h_k(y)}{\lambda_k^2}, \quad x, y \in \R^2. 
\end{equation}

The kernel $K$ has the following regularity properties.
\begin{prop} \label{prop:Kernel} For any $n\in \N^*$, we have 
$K^n \in L^r_{x} \mathcal{W}^{\alpha,2}_{y}$ for any $r \ge 2$ and $\alpha<1-\frac{2}{r}$.
\end{prop}

For the proof of Proposition \ref{prop:Kernel} we will need the following lemmas.

\begin{lem} \label{lem:Kernel}
Let $r,r_1,r_2 \ge 1$, $p\ge 2$, and $0<\alpha<\sigma$ with 
$\frac{1}{r}=\frac{1}{r_1}+\frac{1}{r_2}$ and $\frac{\sigma-\alpha}{2}=\frac{1}{p}$. 
If the functions $f(x,y)$ and $g(x,y)$ $(x,y \in \R^2)$ 
satisfy $f,g \in L^{r_1}_x L^p_y \cap L^{r_2}_{x} \mathcal{W}^{\sigma,2}_{y}$, then $fg \in L^r_x \mathcal{W}^{\alpha,2}_y.$ 
Moreover, 
$$ |fg|_{L^r_x \mathcal{W}^{\alpha,2}_y} 
\lesssim |f|_{L^{r_1}_x L^p_y}|g|_{L^{r_2}_x \mathcal{W}_y^{\sigma,2}} +|f|_{L^{r_2}_x \mathcal{W}_y^{\sigma,2}} |g|_{L^{r_1}_x L^p_y}.$$
\end{lem}

\proof Set $\varepsilon:=\frac{\sigma-\alpha}{2}$, and choose $q_1=p=\bar{q_2}$, $\bar{q_1}=q_2=\frac{2}{1-2\varepsilon}$ in Proposition \ref{prop:dg} (2).
Then, we obtain
\begin{equation} \label{ineq:fg}
|fg (x,\cdot)|_{\mathcal{W}^{\alpha,2}_y} \lesssim |f(x,\cdot)|_{L^p_y}|g(x,\cdot)|_{\mathcal{W}^{\alpha, \frac{2}{1-2\varepsilon}}}
+|f(x,\cdot)|_{\mathcal{W}^{\alpha, \frac{2}{1-2\varepsilon}}}|g(x,\cdot)|_{L^p_y}.
\end{equation}
By Sobolev embedding, $\mathcal{W}^{\sigma,2}(\R^2) \subset \mathcal{W}^{\alpha,\frac{2}{1-2\varepsilon}}(\R^2)$. 
Thus, the right hand side above is majorized by 
$$C(|f(x,\cdot)|_{L^p_y}|g(x,\cdot)|_{\mathcal{W}_y^{\sigma,2}} +|f(x,\cdot)|_{\mathcal{W}_y^{\sigma,2}} |g(x,\cdot)|_{L^p_y}).$$
Now we take the $L^r_x$ norm of  both sides of the inequality, and use H\"older inequality with $\frac{1}{r}=\frac{1}{r_1} +\frac{1}{r_2},$ to get the result.
\hfill\qed
\vspace{3mm}

\begin{cor} \label{cor:Kernel}
Let $n\in \N$, $r,r_1,r_2 \ge 1$, $p\ge 2$, and $0<\tilde{\alpha}<\sigma$ with 
$\frac{1}{r}=\frac{n}{r_1}+\frac{1}{r_2}$ and $\frac{\sigma-\tilde{\alpha}}{2}=\frac{n}{p}$. 
If the function $f(x,y)$ $(x,y \in \R^2)$ 
satisfies $f \in L^{r_1}_x L^p_y \cap L^{r_2}_{x} \mathcal{W}^{\sigma,2}_{y}$, then $f^{n+1} \in L^r_x \mathcal{W}^{\tilde{\alpha},2}_y.$ 
Moreover, 
\begin{equation*} \label{ineq:fpower}
 |f^{n+1}|_{L^r_x \mathcal{W}^{\tilde{\alpha},2}_y} 
\lesssim |f|^n_{L^{r_1}_x L^p_y} |f|_{L^{r_2}_x \mathcal{W}_y^{\sigma,2}}.
\end{equation*}
\end{cor}

\proof We prove the result by induction on $n$. The case $n=1$ is straightforward from the inequality (\ref{ineq:fg}) setting $f=g$, 
combining with H\"older inequality. 
We now assume that the inequality 
$$ |f^{k+1}|_{\mathcal{W}^{\sigma,2}_y} 
\lesssim |f|^k_{L^p_y} |f|_{\mathcal{W}_y^{\tilde{\sigma},2}}. $$
holds for any $\sigma, \tilde{\sigma}$ satisfying $\frac{\tilde{\sigma}-\sigma}{2}=\frac{k}{p}$. 
Applying again Proposition \ref{prop:dg} (2) with $\alpha=\tilde{\alpha}$, $p=2$, $q_1=p$ and $(k+1)\bar{q_2}=p$, we have
\begin{equation} \label{ineq:fk}
|f^{k+1} f|_{\mathcal{W}^{\tilde{\alpha},2}_y} \lesssim |f|_{L^p_y} |f^{k+1}|_{\mathcal{W}^{\tilde{\alpha},\bar{q_1}}_y}+|f|_{\mathcal{W}^{\tilde{\alpha},q_2}_y} |f|_{L^p_y}^{k+1}.
\end{equation}
Note that this choice of parameters implies $\frac{1}{\bar{q_1}}=\frac{1}{2}-\frac{1}{p}$ and $\frac{1}{q_2}=\frac12-\frac{k+1}{p}$. Now
choose $\sigma$ such that $\sigma=1+\tilde{\alpha}-\frac{2}{\bar{q_1}}=\tilde \alpha +\frac{2}{p}$, and $\frac{\tilde{\sigma}-\sigma}{2}=\frac{k}{p}$. 
By Sobolev embedding $\mathcal{W}^{\sigma,2} \subset \mathcal{W}^{\tilde{\alpha},\bar{q_1}}$, 
and by the induction assumption,  
$$ |f^{k+1}|_{\mathcal{W}^{\tilde{\alpha},\bar{q_1}}_y} \lesssim  |f^{k+1}|_{\mathcal{W}^{\sigma,2}_y} \lesssim |f|^k_{L^p_y} |f|_{\mathcal{W}^{\tilde{\sigma},2}_y}.$$
Next, the relation $\frac{\tilde{\sigma}-\tilde{\alpha}}{2}=\frac{k+1}{p}$ implies $\tilde{\sigma}-1=\tilde{\alpha}-\frac{2}{q_2}$, thus 
by Sobolev embedding, 
$|f|_{\mathcal{W}^{\tilde{\alpha},q_2}_y} \lesssim |f|_{\mathcal{W}_y^{\tilde{\sigma},2}}.$
In summary, we estimate the right hand side of (\ref{ineq:fk}) by 
$ C|f|_{L^p_y}^{k+1}|f|_{\mathcal{W}_y^{\tilde{\sigma},2}}$
with $\frac{\tilde{\sigma}-\tilde{\alpha}}{2}=\frac{k+1}{p}$, this proves the $n=k+1$ case. Finally, as in Lemma \ref{lem:Kernel}, 
we use H\"older inequality to conclude the statement.  
\hfill \qed
\vspace{3mm}

\noindent
{\em Proof of Proposition \ref{prop:Kernel}.} 
First of all, we consider the $L^2$ norm in $x_2$ of the kernel $K(x_1,x_2)$, $x_1, x_2 \in \R^2$, on which we apply the cut-off $S_N$; denoting
$\chi_{k,N}=\chi \left[\frac{\lambda_k^2}{\lambda_N^2}\right]$, we may write
$$ S_N K(x_1, x_2) =\sum_{k\in \N^2} \chi_{k,N} \frac{h_k(x_1) h_k(x_2)}{\lambda_k^2}.$$
We have
\begin{eqnarray*}
\left[\int_{\R^2} |S_N K(x_1, x_2)|^2 d x_2 \right]^{\frac12} =\left[\sum_{k\in  \N^2} \chi_{k,N}^2\frac{|h_k(x_1)|^2}{\lambda_k^4} \right]^{\frac12}.  
\end{eqnarray*}
Next we take, for $q>2$, the $L^q$ norm in $x_1$,
$$
|S_N K|_{L^q_{x_1} L^2_{x_2}}^2 = \left[\int_{\R^2} \left[\sum_{k \in \N^2} \chi_{k,N}^2
\frac{|h_k(x_1)|^2}{\lambda_k^4} \right]^{\frac{q}{2}} d x_1 \right]^{\frac{2}{q}} 
\le  \sum_{k \in \N^2} \chi_{k,N}^2 \frac{|h_k|_{L^q}^2}{\lambda_k^4},
$$
where we have used Minkowski inequality.
It follows from (\ref{est:kt}) that $|h_k|_{L^q(\R^2)} \le \lambda_k^{-\theta(q)}$ for some $\theta(q)>0$ if $q>2$, thus 
$$\sum_{k \in \N^2} \chi_{k,N}^2 \frac{|h_k|_{L^q}^2}{\lambda_k^4} 
\le \sum_{k \in \N^2, |k| \le N} \frac{1}{\lambda_k^{4+2\theta(q)}},$$
which converges as $N \to \infty$. Namely $K \in L^q_{x_1} L^2_{x_2}$ for any $q \in (2, +\infty]$.
On the other hand, for $2 \le q <+\infty$, by Sobolev embedding $\mathcal{W}^{1-\frac{2}{q},2}(\R^2) \subset L^q(\R^2)$,  and by duality 
$L^{q*}(\R^2) \subset \mathcal{W}^{1-\frac{2}{q^*},2}(\R^2)$, where $\frac{1}{q}+\frac{1}{q^*}=1$.
We also know that the operator $(-H)^{-1}=(-\Delta+|x|^2)^{-1}$ maps $\mathcal{W}^{1-\frac{2}{q^*},2}(\R^2)$ into 
$\mathcal{W}^{3-\frac{2}{q^*},2}(\R^2) \subset L^{\infty}(\R^2)$. 
Therefore, we have  
$$ \sup_{x_1 \in \R^2} \left|\int_{\R^2} K(x_1, x_2) f(x_2) d x_2 \right|=|(-H)^{-1} f|_{L^{\infty}} \lesssim |(-H)^{-1} f|_{\mathcal{W}^{3-\frac{2}{q^*},2}} 
\lesssim |f|_{\mathcal{W}^{1-\frac{2}{q^*},2}} \lesssim |f|_{L^{q^*}}.$$
This means that there exists $C>0$ which is independent of $x_1$ such that $|K(x_1, \cdot)|_{L^q} \le C$ for any $x_1\in \R^2$, i.e. $K \in L^{\infty}_{x_1} L^q_{x_2}.$ 
Thus, by interpolation between $L^q_{x_1} L^2_{x_2}$, for $q \in (2, +\infty]$, and $L^{\infty}_{x_1} L^q_{x_2}$, for $q \in [2, +\infty)$, we have 
\begin{equation} \label{eq:K1}
K\in L^{\alpha}_{x_1} L^{\beta}_{x_2}, \quad \mathrm{for \; any\;}\alpha >\beta \ge 2.
\end{equation}
The same idea as above can be used to prove that
\begin{eqnarray} \label{eq:K2}
K \in L^{\infty}_{x_1} \mathcal{W}^{\sigma,2}_{x_2}, \quad \mathrm{for \; any\;} \sigma<1.
\end{eqnarray}
Indeed, for any $0<\varepsilon <1$, 
$$ \sup_{x_1 \in \R^2} \left|\int_{\R^2} K(x_1, x_2) f(x_2) d x_2 \right|=|(-H)^{-1} f|_{L^{\infty}} \le C|f|_{\mathcal{W}^{-1+\varepsilon,2}}.$$
By interpolation between (\ref{eq:K1}) and (\ref{eq:K2}), we have 
for all $\sigma<1-\frac{2}{q},$ and $q \ge 2$, 
\begin{equation*}
K \in L^{q}_{x_1} \mathcal{W}^{\sigma,2}_{x_2}. 
\end{equation*}
Now, the use of \eqref{eq:K1} and Corollary \ref{cor:Kernel} allows to obtain the statement of Proposition \ref{prop:Kernel}.
\hfill\qed 
\vspace{3mm}

We note that, as a simple consequence of Proposition \ref{prop:Kernel}, we obtain the following regularity results for $Z_{\infty}^{\gamma_1, \gamma_2}$.
\begin{lem} \label{lem:Kolmogorov} 
Fix any $T>0$. Let $\gamma_1>0, \gamma_2 \in \R, s>0, q \ge 2, sq>2$ and $0<\alpha <\frac 12(s-\frac2q )$. 
The stationary solution $Z_{\infty}^{\gamma_1, \gamma_2}$ of  (\ref{eq:Z}) has a modification in $C^{\alpha}([0,T], \mathcal{W}^{-s,q})$. 
Moreover, there exists a positive constant $C_{T}$ such that 
$$ \E\left[\sup_{t\in [0,T]} |Z^{\gamma_1,\gamma_2}_{\infty}(t) |_{\mathcal{W}^{-s,q}}\right]  \le C_T.$$
\end{lem}

\begin{proof} Since the computation is almost the same as in Lemma 2.1 in \cite{dbdf}, we note only the points in the arguments. 
In order to apply the Kolmogorov test, consider, for $t,\tau$ with $\tau<t \le T$,
\begin{eqnarray*}
Z_{\infty}^{\gamma_1, \gamma_2}(t,x)-Z_{\infty}^{\gamma_1, \gamma_2}(\tau,x)
&=& \sqrt{2\gamma_1} \sum_{k\in \N^2} \Big[\int_{\tau}^t e^{-\lambda_k^2 (\gamma_1+i\gamma_2)(t-\sigma)} d\beta_k(\sigma) h_k(x) \\
             && + \int_{-\infty}^{\tau} (e^{-\lambda_k^2(\gamma_1+i \gamma_2)(t-\sigma)}-e^{-\lambda_k^2(\gamma_1+i\gamma_2)(\tau-\sigma)}) 
             d\beta_k(\sigma) h_k(x) \Big].
\end{eqnarray*}
with $\beta_k=\beta_{k,R} +i \beta_{k,I}$.
Denote by $f_{1,t,\tau}$ the first term in the right hand side above, and by $f_{2,t,\tau}$ the second term.
For any $\alpha\in [0,1]$, and $m \in \N \setminus \{0\}$, since $f_{1,t,\tau}$ is Gaussian, we have the moment estimate:
\begin{eqnarray*}
\E\left[ \Big|(-H)^{-\frac{s}{2}} f_{1,t,\tau}(\cdot,x) \Big|^{2m}\right]^{\frac{1}{2m}} \lesssim (t-\tau)^{\frac{\alpha}{2}} 
\left[\sum_{k\in \N^2} \lambda_k^{-2s+2(\alpha-1)} |h_k(x)|^2 \right]^{\frac12}.
\end{eqnarray*}
We take the $L^q_x$ norm and write the right hand side using the kernel $K(x_1, x_2)$ (see (\ref{def:kernel})):
\begin{eqnarray*}
\left|\E\left[\Big|(-H)^{-\frac{s}{2}}f_{1,t,\tau} \Big|^{2m}\right]^{\frac{1}{2m}} \right|_{L^q_x}
&\le& (t-\tau)^{\frac{\alpha}{2}}  |K|_{L^q_{x_1} \mathcal{W}^{2-s+ (\alpha-1),2}_{x_2}}. 
\end{eqnarray*}
By Proposition \ref{prop:Kernel}, the above norm is finite if $q\ge 2$ and $1-s+\alpha < 1-2/q$. Finally, using Minkowskii inequality
for $2m \ge q$, we obtain 
$$ \E\left( | f_{1,t,\tau} |^{2m}_{\mathcal{W}^{-s,q}} \right) \lesssim (t-\tau)^{m\alpha}, $$
 for $2m \ge q \ge 2$, $0<\alpha< s-2/q.$ The second term can be estimated in a similar way and we obtain  the same estimate
 for $f_{2,t,\tau}$. Applying the Kolmogorov test then gives the result.
\end{proof}

We now use Proposition \ref{prop:Kernel}  to prove the convergence of the Wick products.
\vspace{3mm}

\noindent
{\em Proof of Proposition \ref{prop:ReImCauchy}.} 
In order to lighten the notations, we write in all what follows $Z_{1,N}$ for $S_N Z_{R,\infty}$ and $Z_{2,N}$ for $S_N Z_{I,\infty}$.
We show below that
the sequence $\{: Z_{1,N}^{k}:: Z_{2,N}^{l}:\}_{N\in \N}$ is Cauchy in $L^q(\Omega, \mathcal{W}^{-s,q}(\R^2))$ for any $k,l \in \N$,
and any $q\ge 4$, $s>0$ with $qs>8$.
Let $M, N \in \N$ with $M<N$. Then
\begin{eqnarray*}
\E \left(|: Z_{1,N}^{k}:: Z_{2,N}^{l}:-: Z_{1,M}^{k}:: Z_{2,M}^{l}:|_{\mathcal{W}^{-s,q}}^q \right)
&\le& C\;  \E \left( |(:Z_{1,N}^{k}:-: Z_{1,M}^{k}:):Z_{2,M}^{l}:|_{\mathcal{W}^{-s,q}}^q \right)\\ 
&& \hspace{3mm}+\; C\; \E \left( |:Z_{1,N}^{k}:(:Z_{2,N}^{l}:-:Z_{2,M}^{l}:)|_{\mathcal{W}^{-s,q}}^q \right) \\
&=& (\mathrm{I}) +(\mathrm{II})
\end{eqnarray*}
We only show the estimates for (I) since (II) can be treated similarly. Using Fubini Theorem we may write
\begin{eqnarray} \nonumber 
(\mathrm{I})&=& \int_{\Omega} |(-H)^{-\frac{s}{2}}((:Z_{1,N}^{k}:-: Z_{1,M}^{k}:):Z_{2,M}^{l}:|^q_{L^q} d\prob
\\ \label{set:A}
&=& \int_{\R^2}  \int_{\Omega} \big|\sum_{m\in \N^2}  ((:Z_{1,N}^{k}:-: Z_{1,M}^{k}:):Z_{2,M}^{l}:, h_m)_{L^2} \lambda_m^{-s} h_m(x)\big|^q  d\prob dx.
\end{eqnarray}
We set for $x \in \R^2$,
$$A(x):=\int_{\Omega} \big|\sum_{m\in \N^2}  ((:Z_{1,N}^{k}:-: Z_{1,M}^{k}:):Z_{2,M}^{l}:, h_m)_{L^2} \lambda_m^{-s} h_m(x)\big|^2  d\prob,$$
and for any fixed $N, M\in \N$, 
\begin{equation}
\label{def:kernel_N} K_{N,M} (x_1, x_2) :=\sum_{k\in \N^2}\chi\left(\frac{\lambda_k^2}{\lambda_N^2}\right)
\chi\left(\frac{\lambda_k^2}{\lambda_M^2}\right) \frac{h_k(x_1)h_k(x_2)}{\lambda_k^2} , \quad x_1, x_2 \in \R^2,
\end{equation}
and $K_N(x_1,x_2):=K_{N,N}(x_1,x_2)$.
Then, by the Nelson estimate (\ref{ineq:nelson}), and since
$$ \sum_{m\in \N^2}  ((:Z_{1,N}^{k}:-: Z_{1,M}^{k}:):Z_{2,M}^{l}:, h_m)_{L^2} \lambda_m^{-s} h_m(x) \in \mathcal{H}_{k+l},$$
the right hand side of \eqref{set:A} is bounded above by
$ C_{q, k+l} |A|_{L^{\frac{q}{2}}(\R^2)}^{\frac{q}{2}}$.
Let us first check that $A(x)$ can be written as 
\begin{eqnarray} \label{eq:A}
A(x)=k! \,  l! \, (-H_{x_1})^{-s/2}(-H_{x_2})^{-s/2}[(K_N^k -2K_{N,M}^k+K_M^k)K_M^l](x,x), \quad x \in \R^2.
\end{eqnarray}
Indeed, 
\begin{eqnarray} \nonumber
A(x)&=& \int_{\Omega} \sum_{m_1 \in \N^2} \sum_{m_2 \in \N^2} 
((:Z_{1,N}^{k}:-: Z_{1,M}^{k}:)):Z_{2,M}^{l}:, h_{m_1})_{L^2} \lambda_{m_1}^{-s} h_{m_1}(x) \\ \nonumber
&& \hspace{10mm}\times ((:Z_{1,N}^{k}:-: Z_{1,M}^{k}:)):Z_{2,M}^{l}:, h_{m_2})_{L^2} \lambda_{m_2}^{-s} h_{m_2}(x) d\prob \\ \nonumber
&=& \int_{\Omega} \sum_{m_1,m_2 \in \N^2} \int_{\R^2_{x_1}} \int_{\R^2_{x_2}}
(:Z_{1,N}^{k}:-: Z_{1,M}^{k}:)(x_1) :Z_{2,M}^{l}:(x_1) h_{m_1}(x_1) \lambda_{m_1}^{-s} h_{m_1}(x) \\ \label{ineq1:A}
&& \hspace{10mm}\times (:Z_{1,N}^{k}:-: Z_{1,M}^{k}:)(x_2) :Z_{2,M}^{l}:(x_2) h_{m_2}(x_2) \lambda_{m_2}^{-s} h_{m_2}(x) d x_1 d x_2 d\prob.  
\end{eqnarray}
The fact that $Z_{1,M}$ and $Z_{2,M}$ are independent allows us to compute the integrals on $\Omega$ separately. 
Recalling the definition, 
$$ :Z_{j,N}^{k}:(x_1) = \rho_N (x_1)^k \sqrt{k !} H_k(W^j_{\eta_N(x_1)}), \quad j=1,2,$$
where we use the (obvious) notation $W_f^1$ (resp. $W^2_f$) to denote the Gaussian random variables associated with 
the white noise corresponding to
$Z_1=\mathrm{Re}Z_{\infty}$ (resp. $Z_2=\mathrm{Im} Z_\infty$), we apply (\ref{eq:Wf}): 
\begin{equation} \label{wick1}
\E(H_k(W^j_{\eta_N(x_1)}) H_k(W^j_{\eta_{M}(x_2)}))= k! (\eta_N(x_1),\eta_{M}(x_2))_{L^2}^k, \quad j=1,2.
\end{equation}
Then we have,  by \eqref{ineq1:A},
\begin{eqnarray} \nonumber 
A(x) &=&k! \, l! \sum_{m_1,m_2 \in \N^2} \int_{\R^2_{x_1}} \int_{\R^2_{x_2}}
\Big\{ \rho_N^k(x_1) \rho_N^k(x_2) (\eta_N(x_1),\eta_{N}(x_2))_{L^2}^k 
 - \rho_N^k(x_1) \rho_M^k(x_2) (\eta_N(x_1),\eta_{M}(x_2))_{L^2}^k \\ \nonumber 
&& \hspace{2.5cm}-\rho_M^k(x_1) \rho_N^k (x_2) (\eta_M(x_1),\eta_{N}(x_2))_{L^2}^k 
+  \rho_M^k(x_1) \rho_M^k(x_2) (\eta_M(x_1),\eta_{M}(x_2))_{L^2}^k \Big\} \\ \label{ineq2:A}
&& \hspace{.5cm}\times  \rho_M^l(x_1) \rho_M^l(x_2) (\eta_M(x_1),\eta_{M}(x_2))_{L^2}^l 
h_{m_1}(x_1) h_{m_2}(x_2) \lambda_{m_1}^{-s} \lambda_{m_2}^{-s} h_{m_1}(x) h_{m_2}(x) d x_1 d x_2.
\end{eqnarray}
Again by the definitions \eqref{def:etaN} and \eqref{def:kernel_N},
\begin{eqnarray} \label{wick2}
(\eta_N(x_1),\eta_{N}(x_2))_{L^2}^k &=& \frac{1}{(\rho_N(x_1) \rho_N(x_2))^k} K^k_N(x_1,x_2) \\ \nonumber 
(\eta_N(x_1),\eta_{M}(x_2))_{L^2}^k &=& \frac{1}{(\rho_N(x_1) \rho_M(x_2))^k} K^k_{N,M} (x_1,x_2).
\end{eqnarray}
We thus have by \eqref{ineq2:A}
\begin{eqnarray*}
&&A(x)
= k!\, l!\,\sum_{m_1 \in \N^2} \lambda_{m_1}^{-s} h_{m_1}(x) \\ 
&& \times \int_{\R^2_{x_1}} \left[\sum_{m_2 \in \N^2} \int_{\R^2_{x_2}} (K_N^k(x_1,x_2)-K_M^k(x_1,x_2)) K_M^l(x_1,x_2) h_{m_2}(x_2) d x_2 \lambda_{m_2}^{-s} h_{m_2}(x) \right]
h_{m_1}(x_1) d x_1;  
\end{eqnarray*}
note that the term in square brackets is equal to 
$$(-H_{x_2})^{-\frac{s}{2}}\big(K_N^k(x_1,x)-2K_{N,M}^k(x_1,x) 
+K_M^k(x_1,x)\big)K_M^l(x_1,x).$$
Using the same computations for the $x_1$ variable leads to \eqref{eq:A}.

Now, let
$$
\tilde{A}(\eta,\xi):= k! \,l! \,(-H_{\eta})^{-\frac{s}{2}}(-H_{\xi})^{-\frac{s}{2}}[(K_N^k -2K_{N,M}^k+K_M^k)K_M^l](\eta,\xi), \quad \eta, \xi \in \R^2.\\
$$
By the Sobolev embedding $\mathcal{W}^{1+\frac{s}{2},2}(\R^2) \subset L^{\infty}(\R^2)$, we have
$$|\tilde{A}(\eta, \xi)| \lesssim |(-H_\eta)^{-\frac{s}{2}}[(K_N^k-2K_{N,M}^k+K_M^k)K_M^l](\eta,\cdot)|_{\mathcal{W}_{\xi}^{1-\frac{s}{2},2}}.$$
Thus  
\begin{eqnarray}
\label{ineq:Aq}
\int_{\R^2_x} (A(x))^{\frac{q}{2}} dx &=& \int_{\R^2_x} (\tilde{A}(x,x))^{\frac{q}{2}} dx \nonumber\\ 
 &\le& \int_{\R^2_x} |(-H_x)^{-\frac{s}{2}}[(K_N^k-2K_{N,M}^k+K_M^k)K_M^l](x,\cdot)|^{\frac{q}{2}}_{\mathcal{W}_{\xi}^{1-\frac{s}{2},2}} dx \nonumber\\
 &\le&  |(K_N^k-2K_{N,M}^k+K_M^k)K_M^l|^{\frac{q}{2}}_{L_{\eta}^{\frac{q}{2}} \mathcal{W}_{\xi}^{1-\frac{s}{2},2}} \nonumber\\
 & \lesssim & |(K_N^k-K_{N,M}^k)K_M^l|^{\frac{q}{2}}_{L_{\eta}^{\frac{q}{2}} \mathcal{W}_{\xi}^{1-\frac{s}{2},2}}
+ |(K_{N,M}^k-K_M^k)K_M^l|^{\frac{q}{2}}_{L_{\eta}^{\frac{q}{2}} \mathcal{W}_{\xi}^{1-\frac{s}{2},2}},
\end{eqnarray}
by the embedding $L^{\frac{q}{2}}(\R^2) \subset \mathcal{W}^{-s, \frac{q}{2}}(\R^2)$. Since 
$$K_N^k-K_{N,M}^k=(K_N-K_{N,M}) \sum_{j=0}^{k-1} K_N^j K_{N,M}^{k-1-j},$$
applying three times Lemma \ref{lem:Kernel}, and thanks to the Sobolev embeddings $\mathcal{W}^{1-\frac{s}{4},2}\subset L^p$, 
$\mathcal{W}^{1-\frac{s}{8},2}\subset L^{2p}$, and $\mathcal{W}^{1-\frac{s}{16},2}\subset L^{4p}$, with $p=\frac{8}{s}$, we may estimate the first term in the right hand side of \eqref{ineq:Aq} 
by
\begin{eqnarray*}
&&\Big[|K_N-K_{N,M}|_{L_{\eta}^{q} \mathcal{W}_{\xi}^{1-\frac{s}{4},2}} \sum_{j=0}^{k-1}
|K_N^j|_{L_{\eta}^{2q} \mathcal{W}_{\xi}^{1-\frac{s}{8},2}} |K_{N,M}^{k-1-j}|_{L_{\eta}^{4q} \mathcal{W}_{\xi}^{1-\frac{s}{16},2}}
|K_M^{l}|_{L_{\eta}^{4q} \mathcal{W}_{\xi}^{1-\frac{s}{16},2}} \Big]^{\frac{q}{2}}\\
&\le & \left[\int_{\R^2} \Big[ \sum_{|j|=\frac{M}{2}}^N \lambda_j^{-2+2(1-\frac{s}{4})} h_j^2(x) \Big]^{\frac{q}{2}} dx \right]^{\frac{1}{2}}
\Big[\sum_{j=0}^{k-1}
|K^j|_{L_{\eta}^{2q} \mathcal{W}_{\xi}^{1-\frac{s}{8},2}} |K^{k-1-j}|_{L_{\eta}^{4q} \mathcal{W}_{\xi}^{1-\frac{s}{16},2}}
|K^{l}|_{L_{\eta}^{4q} \mathcal{W}_{\xi}^{1-\frac{s}{16},2}} \Big]^{\frac{q}{2}}.
\end{eqnarray*}
Choose $\delta>0$ such that $\delta <\frac{s}{4}$, and we have 
$$
\left[\int_{\R^2} \Big( \sum_{|j|=\frac{M}{2}}^N \lambda_j^{-\frac{s}{2}} h_j^2(x) \Big)^{\frac{q}{2}} dx \right]^{\frac{1}{2}} 
\le \lambda_M^{-\frac{\delta q}{2}} |K|^q_{L_{\eta}^{q} \mathcal{W}_{\xi}^{1-\frac{s}{4}+\delta,2}}.
$$
Thus, recalling $\lambda_M=O(\sqrt{M})$, using Proposition \ref{prop:Kernel}, and estimating the second term in the right hand side 
of \eqref{ineq:Aq} in the same way, we conclude to the convergence of $(I)$ as $N,M$ tend to infinity,
provided $\delta$ may be chosen such that  $1-\frac{s}{4}+\delta < 1-\frac{2}{q}$,   
i.e. $qs >8$, $q > 2$, $s >0$. This shows that the sequence $\{: Z_{1,N}^{k}:: Z_{2,N}^{l}:\}_{N\in \N}$ is a Cauchy sequence in 
$L^q(\Omega, \mathcal{W}^{-s,q}(\R^2))$ for any $k,l \in \N$. The bound \eqref{bound:Z} is obtained by estimating
$\E [|:(S_N Z_{R,\infty})^{k}: :(S_N Z_{I,\infty})^{l}:|^q_{\mathcal{W}^{-s,q}}]$ uniformly in $N$, using the same arguments as above.
\hfill\qed
\vspace{3mm}

\subsection{Properties of the semigroup}
The purpose of this section is to give some lemmas listing useful properties of the semigroup $(e^{t(\gamma_1+i\gamma_2)H})_{t\ge 0}$, that will later be used in the course of the proofs of the theorems.

\begin{lem} \label{lem:heat_kernel}
Let $\gamma_1>0$ and $\gamma_2 \in \R$. There exists a constant $C_0=C_0(\gamma_1, \gamma_2)>0$ such that for $t>0$ satisfying $\gamma_1 t$, $|\gamma_2 t|<<1$,  
\begin{equation*}
|e^{t (\gamma_1+i\gamma_2) H} f|_{L^r}\le C_0 t^{-\frac{1}{l}}|f|_{L^s},
\end{equation*}
for any $f\in L^s(\R^2)$, with 
$$ 
0 \le \frac{1}{r} \le \frac{1}{r} +\frac{1}{l}=\frac{1}{s} \le 1.
$$
\end{lem}
\begin{proof}
The proof of Proposition 3.1 of \cite{dbdf} can be applied with minor modifications.
\end{proof}
\vspace{3mm}

The proof of the next lemmas will be given in Appendix 1.

\begin{lem} \label{lem:regularization} 
Let $\gamma_1>0$, $\gamma_2 \in \R$, $1< p <\infty$ and $s \in (0,2]$. 
Then there exists a constant $C_1=C_1(\gamma_1, \gamma_2)>0$ such that , 
\begin{equation} \label{ineq:regularization} 
|e^{t(\gamma_1+i\gamma_2) H} f |_{\mathcal{W}^{s,p}} \le C_1 t^{-\frac{s}{2}} |f|_{L^p},  
\end{equation}
for $t>0$ satisfying $\gamma_1 t$, $|\gamma_2 t|<<1$, and for any $f\in L^p(\R^2)$.
\end{lem}
\begin{rem}
In the purely dissipative case, i.e. the case $\gamma_2=0$, the above estimate holds for all $s> 0$.
\end{rem}

\begin{lem} \label{lem:arnaud} 
Let $\gamma_1>0$, $\gamma_2 \in \R$, $p> 2$ and $\beta > \frac12 -\frac1p $. For $t>0$ satisfying $\gamma_1 t$, $|\gamma_2 t|<<1$, 
$$
|e^{t (\gamma_1+i\gamma_2) H} f|_{L^2} \le C_{\gamma_1, \gamma_2} t^{-\beta}  |f|_{L^p},
$$
for any $f \in L^p(\R^2)$.
\end{lem}

The following lemma is a consequence of Lemmas \ref{lem:heat_kernel} and \ref{lem:arnaud}, and an interpolation argument.
\begin{lem} \label{lem:arnaud2}
Let $\gamma_1>0$, $\gamma_2 \in \R$, $q>p> 2$ and $\sigma >\frac1p-\frac1q$. Then there exists a constant $C_2=C_2(\gamma_1, \gamma_2)>0$ such that  
for $t>0$ satisfying $\gamma_1 t$, $|\gamma_2 t| <<1$,
$$
|e^{t(\gamma_1+ i\gamma_2) H} f|_{L^p} \le C_2 t^{-\sigma}  |f|_{L^q},
$$
for any $f\in L^q(\R^2)$.
\end{lem}
\vspace{3mm}

We still need some preliminaries before proving Theorem \ref{thm:localexistence}. 
Let $T>0$ be fixed so that the above semigroup estimates are valid in the time interval $[0,T]$. 
Let $s,\beta>0$ and $r,p,q \ge 1$.  
We consider the function space 
\begin{equation*}
\mathcal{E}_T:=C([0,T]; \mathcal{W}^{-s,q})\cap L^r(0,T; \mathcal{W}^{\beta,p}). 
\end{equation*}

\begin{lem} \label{lem:data} Let $\gamma_1, \beta, s>0$, $\gamma_2 \in \R$, $q>p>2$ and $\sigma>\frac1p-\frac1q$. 
Assume $\varepsilon:=-\sigma-\frac{\beta+s}{2}+\frac{1}{r}>0.$ 
Then, there exists a constant $C_3=C_3(\gamma_1,\gamma_2,T)$ such that for any $f \in \mathcal{W}^{-s,q}$,  
\begin{equation*}
t \mapsto e^{t (\gamma_1+i\gamma_2) H} f  \in \mathcal{E}_T, 
\end{equation*}
and 
$$
|e^{\cdot (\gamma_1+i\gamma_2) H} f |_{\mathcal{E}_T} \le C_3 |f|_{ \mathcal{W}^{-s,q}}.
$$
\end{lem}

\begin{proof}
First, by Lemma \ref{lem:heat_kernel}, for any $t \in [0,T]$,
\begin{eqnarray*}
|e^{t(\gamma_1+i\gamma_2) H} f|_{\mathcal{W}^{-s,q}} \le C_0 |f|_{\mathcal{W}^{-s,q}}. 
\end{eqnarray*}
Moreover, the strong continuity of the semi-group $t \mapsto e^{t(\gamma_1+i\gamma_2) H}$ in $L^2(\R^2)$ and the above bound easily
shows that $t \mapsto e^{t(\gamma_1+i\gamma_2) H} f$ is continuous with values in $\mathcal{W}^{-s,q}$.
Using Lemmas \ref{lem:arnaud2} and \ref{lem:regularization}, we have
\begin{eqnarray*}
  |e^{t(\gamma_1+ i\gamma_2) H} f|_{L^r(0,T;\mathcal{W}^{\beta,p})} 
  &\le& C_1 C_2 \left[\int_0^T  \Big[t^{-\sigma-\frac{\beta+s}{2}} |f|_{\mathcal{W}^{-s,q}} \Big]^r  dt \right]^{1/r}\\
  &\le & C_1 C_2  T^{-\sigma-\frac{\beta+s}{2}+\frac{1}{r}} |f|_{\mathcal{W}^{-s,q}},
\end{eqnarray*}
which ends the proof.
\end{proof}

\begin{lem} \label{lem:nonlinear} Let $\gamma_1,s>0$, $\gamma_2 \in \R$, $l\in \N$, $q>p > 2$, $0<\beta<2/p$, $\sigma>\frac1p-\frac1q$, and $r\ge l+1$
be such that $ \varepsilon :=-\sigma-\frac{\beta+s}{2}+\frac{1}{r}>0$ and 
$\delta:= \frac{l}{2} (\beta-\frac{2}{p})-\frac{l+1}{r}+1>0$.
Let $f\in L^{\frac{r}{l+1}}(0,T; \mathcal{W}^{-\alpha,q})$, where 
$\alpha=s+l(\frac2p-\beta).$ 
Then,
\begin{equation*}
t \mapsto \int_0^t e^{(t-\tau) (\gamma_1+i\gamma_2) H} f(\tau) d\tau \in \mathcal{E}_T 
\end{equation*}
with 
\begin{equation*}
\sup_{t\in [0,T]}\Big|\int_0^t e^{(t-\tau)(\gamma_1+i\gamma_2) H} f(\tau) d\tau\Big|_{\mathcal{W}^{-s,q}}
\le C_1 T^{\delta} |f|_{L^{\frac{r}{l+1}}(0,T; \mathcal{W}^{-\alpha,q})},
\end{equation*}
and 
\begin{equation*}
\Big|\int_0^\cdot e^{(\cdot-\tau)(\gamma_1+ i\gamma_2) H} f(\tau) d\tau\Big|_{L^r(0,T; \mathcal{W}^{\beta,p})} 
\le \tilde C_l T^{\varepsilon+\delta}
|f|_{L^{\frac{r}{l+1}}(0,T; \mathcal{W}^{-\alpha,q})}.
\end{equation*}
\end{lem}

\begin{proof} By Lemma \ref{lem:regularization}, if $t \in [0,T]$,
\begin{eqnarray*}
\Big|\int_0^t  e^{(t-\tau)(\gamma_1+i\gamma_2) H} f(\tau) d\tau\Big|_{\mathcal{W}^{-s,q}}
&\le& C_1 \int_0^t (t-\tau)^{-\frac{\alpha-s}{2}}|f(\tau)|_{\mathcal{W}^{-\alpha,q}} d\tau \\
&\le& C_1 T^{\delta} |f|_{L^{\frac{r}{l+1}}(0,T; \mathcal{W}^{-\alpha,q})},
\end{eqnarray*} 
where we have used the H\"older inequality with $\frac{1}{r'}+\frac{l+1}{r}=1$ in the second inequality. The continuity in time with values in
$\mathcal{W}^{-s,q}$ follows from the strong continuity of the semi-group and the above estimate.
Also, applying Lemma \ref{lem:regularization} and Lemma \ref{lem:arnaud2},
\begin{eqnarray*} 
\Big|\int_0^t e^{(t-\tau) (\gamma_1+i\gamma_2) H} f(\tau) d\tau\Big|_{\mathcal{W}^{\beta,p}} 
&\le& \tilde C_l \int_0^t (t-\tau)^{-\frac{\beta+\alpha}{2}} |e^{\frac12 {(t-\tau)(\gamma_1+i\gamma_2) H}} f(\tau)|_{\mathcal{W}^{-\alpha,p}} d\tau \\
&\le & \tilde C_l \int_0^t (t-\tau)^{-\frac{\beta+\alpha}{2}} (t-\tau)^{-\sigma} |f(\tau)|_{\mathcal{W}^{-\alpha,q}} d\tau. 
\end{eqnarray*}
Haussdorf-Young inequality with $1+\frac{1}{r}=\frac{1}{\tilde{\gamma}}+\frac{l+1}{r}$ implies
\begin{eqnarray*}
 \Big|\int_0^t e^{(t-\tau)(\gamma_1+i\gamma_2) H} f(\tau) d\tau\Big|_{L^r(0,T; \mathcal{W}^{\beta,p})} 
 &\le& \tilde C_l \,T^{-\frac{\alpha+\beta}{2}-\sigma +\frac{1}{\tilde{\gamma}}}
 |f|_{L^{\frac{r}{l+1}}(0,T; \mathcal{W}^{-\alpha,q})},
\end{eqnarray*}
where we note that $-\frac{\alpha+\beta}{2}-\sigma +\frac{1}{\tilde{\gamma}} = \varepsilon+\delta >0$. 
\end{proof}

\subsection{Local existence}
In order to handle the nonlinear term, we will need the following estimate on $\mathcal{W}^{s,p}$ norms of products of functions.

\begin{lem} \label{lem:bilinear} Let $q>p>2$, $0<s<\beta<2/p$, and $m \in \N^*$. 
Suppose 
$\beta-s-(m-1)(\frac{2}{p}-\beta) > 0$, and  $s+m(\frac{2}{p}-\beta)<2(1-\frac{1}{q})$.
Then, there is a constant $C>0$ such that 
\begin{equation*}
|h f^{m}|_{\mathcal{W}^{-(s+m(\frac{2}{p}-\beta)), q}} \le C|h|_{\mathcal{W}^{-s,q}} |f|_{\mathcal{W}^{\beta,p}}^{m}. 
\end{equation*}
\end{lem}
\vspace{3mm}

The proof of Lemma \ref{lem:bilinear} will be found in Appendix 2.
\begin{rem} \label{rem:bilinear}
The conditions on the parameters are typically satisfied for $q$ very large, $s$ small and $\beta$ close to $\frac{2}{p}$. 
By H\"older inequality, we infer from Lemma \ref{lem:bilinear} that
\begin{equation*}
|h f^{m}|_{L^{\frac{r}{m+1}}(0,T;\mathcal{W}^{-(s+m(\frac{2}{p}-\beta)), q})} \le C|h|_{L^r(0,T;\mathcal{W}^{-s,q})} |f|^m_{L^r(0,T;\mathcal{W}^{\beta,p})}. 
\end{equation*}
We will use this estimate to control the nonlinearity, the cubic nonlinear power corresponds to the case $m\le2$ in Lemma 
\ref{lem:bilinear}.  
\end{rem}

\begin{rem}
\label{rem:parameters}
Let us show that we can find parameters satisfying the conditions in Lemmas \ref{lem:data} to \ref{lem:bilinear}
when we consider more general polynomial nonlinearities with an even integer $m\in \N^*$:
\begin{equation*} 
dX=(\gamma_1+i\gamma_2)(HX-:|X|^{m} X:)dt +\sqrt{2\gamma_1}dW.  
\end{equation*}
Let $p>3r,$ with $r>2(m+1).$ Then, 
$\frac{2}{p}- \frac{1}{m} <0<\frac{2}{p}< \frac{1}{r}-\frac{1}{p}$.
Since 
$\frac{2}{p}- \frac{1}{m}<\frac{2}{r}-\frac{1}{m}<\frac{1}{m+1}-\frac{1}{m}<0$,
we can then take $s$ and $\beta$ such that 
$\frac{2}{p}- \frac{1}{m} <0 <s< \beta<\frac{2}{p} < \frac{1}{r}-\frac{1}{p}$.
Thus, $\frac{2}{p}-\beta<\frac{1}{m}$, so that
\begin{eqnarray*}
\delta = \frac{m}{2}\big(\beta-\frac{2}{p}\big)-\frac{m+1}{r}+1 
= 1-\frac{m}{2}\big(\frac{2}{p}-\beta\big)-\frac{m+1}{r} 
> \frac{1}{2}-\frac{m+1}{r}>0.
\end{eqnarray*}
Next, take $\sigma=\frac{1}{p}$. Then we have 
$\varepsilon = \frac{1}{r}-\frac{1}{p} -\frac{\beta+s}{2}
> \frac{1}{r}-\frac{1}{p}-\beta >0$.
In addition, $\beta$ and $s$ should satisfy the conditions in Lemma \ref{lem:bilinear}, and $q$ should satisfy $qs>8$, but
this is possible if we take $s$ very small, $\beta$ close to $\frac{2}{p}$ and $q>p$ very large. Note that
the assumptions in Theorem \ref{thm:localexistence} allow us to use Lemmas \ref{lem:heat_kernel} to \ref{lem:bilinear} when $m=2$.  
\end{rem}
\vspace{3mm}

We now have all the estimates in hand to solve equation (\ref{eq:u}) in the mild form: for any fixed $T>0$, we consider for $t\in [0,T]$ :
\begin{eqnarray}
\label{eq:umild}
u(t)= e^{t (\gamma_1+i\gamma_2) H}u_0-(\gamma_1+i\gamma_2)\int_0^t e^{(t-\tau)(\gamma_1+ i\gamma_2) H} :|u+Z_{\infty}^{\gamma_1, \gamma_2}|^2 (u+Z_{\infty}^{\gamma_1, \gamma_2})(\tau): d\tau.   
\end{eqnarray}
By the definition of the Wick product, we may write, using again the notation $Z_{R, \infty}=\mathrm{Re}\,(Z_{\infty}^{\gamma_1, \gamma_2})$, 
$Z_{I, \infty}=\mathrm{Im}\,(Z_{\infty}^{\gamma_1, \gamma_2})$, $u_R=\mathrm{Re} \,u$, and $u_I=\mathrm{Im} \,u$, 
\begin{equation}
\label{def:F}
:|u+Z_{\infty}^{\gamma_1, \gamma_2}|^2 (u+Z_{\infty}^{\gamma_1, \gamma_2}): 
= F\left(u, (:Z_\infty^l:)_{1\le l \le 3}\right) =F_0 +F_1+F_2+F_3
\end{equation}
with $F_0 =|u|^2u$, and
\begin{eqnarray*}
F_1 &= &Z_\infty |u|^2 +2 Z_{R,\infty} u_R u+2 Z_{I,\infty} u_I u,\\
F_2 &= &: Z_{R,\infty}^2: (3u_R+iu_I) \, + :Z_{I,\infty}^2: (u_R + 3iu_I) +2:Z_{R,\infty}Z_{I,\infty}:(u_I+iu_R),\\
F_3& = & :Z_{R,\infty}^3: +\, i:Z_{I,\infty}^3: + :Z_{R,\infty} Z_{I,\infty}^2: + \,i:Z_{R,\infty}^2 Z_{I,\infty}: .
\end{eqnarray*}
\vspace{3mm}

\noindent
{\em Proof of Theorem \ref{thm:localexistence}.} We fix all the parameters appearing in lemmas 3.9 and 3.10, as in Remark~\ref{rem:parameters}
(with $m=2$).
Let us fix any $T>0$, and take $T_0 \le T \wedge 1$ small enough so that 
the estimates in Lemma \ref{lem:data} and Lemma \ref{lem:nonlinear} hold. 
We consider the closed ball in $\mathcal{E}_{T_0}$ with radius $R>0$:
$$B_R(T_0) := \{u \in \mathcal{E}_{T_0}, |u|_{L^{\infty}(0,T_0, \mathcal{W}^{-s,q}) \cap L^r(0,T_0, \mathcal{W}^{\beta,p})}  \le R\},$$
where $R$ is such that
$$
C_3|u_0|_{\mathcal{W}^{-s,q}}
 +(\gamma_1 +|\gamma_2|)(C_0 +\tilde C_0) |F_3|_{L^r(0,T;\mathcal{W}^{-s,q})}=\frac{R}{2}.
 $$
 Here, the constants $C_0$ and $\tilde C_0$ are the constants appearing in Lemmas \ref{lem:heat_kernel} and \ref{lem:nonlinear}.  Note that $R$ is random, but 
a.s. finite, as follows from Proposition \ref{prop:ReImCauchy} and H\"older's inequality.
Consider the mapping, for $t \le T_0$, and $u \in B_R(T_0)$, 
$$(\mathcal{T}u) (t) = e^{t(\gamma_1+ i\gamma_2) H}u_0 -(\gamma_1+i\gamma_2) 
\int_0^t e^{(t-\tau)(\gamma_1+i \gamma_2) H} :|u+Z_{\infty}^{\gamma_1, \gamma_2}|^2 (u+Z_{\infty}^{\gamma_1, \gamma_2}) (\tau): d\tau.$$
First of all,  we estimate the linear part. By Lemma \ref{lem:data}, we have 
\begin{equation*}
 |e^{t(\gamma_1+i\gamma_2)H} u_0|_{\mathcal{E}_T} \le C_3 |u_0|_{\mathcal{W}^{-s,q}}.
\end{equation*}
Next, we estimate the nonlinear part, using the formulation \eqref{def:F}, and we start with the term involving $F_3$.
It follows from Lemma \ref{lem:heat_kernel} that, for all $t\in [0,T_0]$,
\begin{eqnarray*}
\Big|(\gamma_1+i\gamma_2)\int_0^t e^{(t-\tau)(i+ \gamma) H} F_3(\tau) d\tau\Big|_{\mathcal{W}^{-s,q}} 
&\le& (\gamma_1+|\gamma_2|)C_0 \int_0^t |F_3(\tau)|_{\mathcal{W}^{-s,q}} d\tau \\
&\le& (\gamma_1+|\gamma_2|) C_0 T_0^{1-\frac{1}{r}} |F_3|_{L^r(0,T;\mathcal{W}^{-s,q})}.
\end{eqnarray*}
We next apply Lemma \ref{lem:nonlinear} with $l=0$, then with $\delta_0 := -\frac{s+\beta}{2} +1-\frac{1}{r} >0$, we have,  
\begin{eqnarray*}
\Big|(\gamma_1+i\gamma_2)\int_0^t e^{(t-\tau)(\gamma_1 + i\gamma_2) H} F_3(\tau) d\tau\Big|_{L^r(0,T_0;\mathcal{W}^{\beta,p})}
&\le& (\gamma_1+|\gamma_2|) \tilde C_0 T_0^{\delta_0} |F_3|_{L^r(0,T;\mathcal{W}^{-s,q})}. \\
\end{eqnarray*}
Let us  now turn to the estimate of the term containing $F_2$. Applying Lemma \ref{lem:nonlinear} with $l=1$, we have
\begin{eqnarray*}
&&\Big|(\gamma_1+i\gamma_2)\int_0^t e^{(t-\tau)(\gamma_1+ i\gamma_2) H} F_2(\tau) d\tau\Big|_{L^{\infty}(0,T_0; \mathcal{W}^{-s,q}) \cap L^r(0,T_0;\mathcal{W}^{\beta,p})} \\
&\le & (\gamma_1+|\gamma_2|) (C_1+\tilde C_1)  T_0^{\delta_1} |F_2|_{L^{\frac{r}{2}}(0,T_0;\mathcal{W}^{-(s+\frac{2}{p}-\beta),q})},
\end{eqnarray*}
with $\delta_1:= 1-\frac{2}{r}-\frac 12 (\frac 2p -\beta)>0$; 
the use of formula \eqref{def:F} and Lemma \ref{lem:bilinear} with $m=1$ allows then to bound the above term by
\begin{eqnarray*}
& &3 C(\gamma_1+|\gamma_2|) \,(C_1+\tilde C_1) T_0^{\delta_1} \sum_{l=0}^2  |u|_{L^r(0,T_0; \mathcal{W}^{\beta,p})} |:Z_{R,\infty}^{l}
Z_{I,\infty}^{2-l}:|_{L^r(0,T_0;\mathcal{W}^{-s,q})}\\
&\le &3C (\gamma_1+|\gamma_2|) \, (C_1+\tilde C_1) T_0^{\delta_1} R \sum_{l=0}^2 |:Z_{R,\infty}^{l}
Z_{I,\infty}^{2-l}:|_{L^r(0,T_0;\mathcal{W}^{-s,q})}.
\end{eqnarray*}
In the same way, the term involving $F_1$ is estimated thanks to Lemma \ref{lem:nonlinear} and Lemma \ref{lem:bilinear}, respectively with
$l=2$ and $m=2$; we obtain
\begin{eqnarray*}
&&\Big|(\gamma_1+i\gamma_2)\int_0^t e^{(t-\tau)(\gamma_1 + i\gamma_2) H} F_1(\tau) d\tau\Big|_{L^{\infty}(0,T_0; \mathcal{W}^{-s,q}) \cap L^r(0,T_0;\mathcal{W}^{\beta,p})} \\
&\le& (\gamma_1+|\gamma_2|)\,(C_1+ \tilde C_2) T_0^{\delta_2} |F_1(\tau)|_{L^{\frac r3}(0,T_0;\mathcal{W}^{-(s+2(\frac{2}{p}-\beta)),q})} \\
&\le&3 C (\gamma_1+|\gamma_2|)\,  (C_1+\tilde C_2) \,T_0^{\delta_2} (|Z_{R,\infty}|_{L^r(0,T_0;\mathcal{W}^{-s,q})} 
+|Z_{I,\infty}|_{L^r(0,T_0;\mathcal{W}^{-s,q})} ) |u|^2_{L^r(0,T_0; \mathcal{W}^{\beta,p})}\\ 
&\le& 3C (\gamma_1+|\gamma_2|) \,(C_1+\tilde C_2) \,R^2 \,T_0^{\delta_2} |Z_\infty|_{L^r(0,T;\mathcal{W}^{-s,q})},
\end{eqnarray*}
with $\delta_2:=1-\frac 3r -(\frac 2p -\beta) >0$.
Finally, using again Lemma \ref{lem:nonlinear} and Lemma \ref{lem:bilinear} with $l=m=2$ leads to
\begin{eqnarray*}
&&\left|(\gamma_1+i\gamma_2) \int_0^t e^{(t-\tau)(\gamma_1+i\gamma_2) H} (|u|^2u)(\tau) d\tau \right|_{L^{\infty}(0,T_0; \mathcal{W}^{-s,q}) \cap L^r(0,T_0;\mathcal{W}^{\beta,p})} \\
&\le& (\gamma_1+|\gamma_2|) C(C_1+\tilde C_2) \,T_0^{\delta_2} |u|_ {L^r(0,T_0;\mathcal{W}^{-s,q})}
|u|^2_{L^r(0,T_0; \mathcal{W}^{\beta,p})}\\ 
&\le& (\gamma_1+|\gamma_2|) C(C_1+\tilde C_2) \,T_0^{\delta_2} \, R^3.
\end{eqnarray*}

Gathering all the above estimates, and since $T_0 \le 1$, we have for some constants $C>0$,
$$
|\mathcal{T}u|_{\mathcal{E}_{T_0}} \le
\frac{R}{2} + C T_0^{\delta_2}R(1+R^2) \sum_{0\le l+k\le 3} |:Z_{R,\infty}^l Z_{I,\infty}^k:|_{L^r(0,T;\mathcal{W}^{-s,q})}.
$$
By Proposition \ref{prop:ReImCauchy} and H\"older's inequality, it follows that the right hand side is smaller than $R$ provided
$T_0$ is chosen such that
\begin{equation}
\label{eq:extime}
C T_0^{\delta_2}(1+R^2) \sum_{0\le l+k\le 3} |:Z_{R,\infty}^l Z_{I,\infty}^k:|_{L^r(0,T;\mathcal{W}^{-s,q})}\le \frac 12,
\end{equation}
and $\mathcal{T}$ maps the ball $B_R(T_0)$ into itself. Estimating in the same way as above the difference $|\mathcal{T}u_1 - \mathcal{T}u_2|_{\mathcal{E}_{T_0}}$, it can easily be checked that $\mathcal{T}$ is a contraction on $B_R(T_0)$ under the condition \eqref{eq:extime}.

By classical PDE arguments, the continuity of the solution $u$ with respect to the initial data $u_0$, and the following alternative also hold:
Let $T^*_0$ be the maximal time $T^*_0(\omega):=\sup\{T_0 \ge 0, \exists~ \mbox{unique solution}~  u\in \mathcal{E}_{T_0} \}$. 
Then, almost surely, $T^*_0=+\infty$ or $T^*_0< +\infty$ and 
$$ \lim_{t \uparrow T^*_0} |u(t)|_{\mathcal{W}^{-s,q}}=+\infty.$$
\hfill \qed

\section{Long time behaviour in the large dissipation case} 
When the dissipation parameter $\gamma_1$ is sufficiently large compared to $\gamma_2$,
an energy bound allows us to globalize the solution. This is the object of the first part of the section.
In a second part, we use another energy bound to prove the existence of an invariant measure, still in the large dissipation
case.

\subsection{Global existence}
Before proving the global existence result, we start with the proof of Proposition~\ref{prop:LpLWP}, which gives another version of local existence
result for initial data in $L^q(\R^2)$, under slightly more restrictive conditions on the parameters.
\vspace{3mm}

\noindent
{\em Proof of Proposition \ref{prop:LpLWP}}.
Let $u_0 \in L^q(\R^2)$, and $u\in C([0,T_0^*);{\mathcal{W}}^{-s,q})\cap L^r(0,T_0^{*-};{\mathcal{W}}^{\beta,p})$ be the solution of \eqref{eq:umild} given by Theorem \ref{thm:localexistence}. First, it is clear from Lemma
\ref{lem:data} that $t\mapsto e^{t(\gamma_1 +i\gamma_2)H} u_0$ is in $C([0,T_0^*); L^q(\R^2))$. On the other hand, applying Lemma \ref{lem:nonlinear}
with $s=\frac 2p -\beta$, an integer $l$ such that $1\le l \le 3$, and $\alpha = (l+1)( \frac 2p -\beta)$, and noticing moreover that the condition
$\frac l2 ( \beta-\frac 2p ) - \frac{l+1}{r} +1>0$ is satisfied for $r>6$ since $\frac 2p -\beta <\frac 29$, we obtain, for all
$f \in L^{\frac{r}{l+1}} (0,T;{\mathcal{W}}^{-l(\frac 2p -\beta),q})$, with $T<T_0^*$ and all $t\in [0,T]$,
\begin{eqnarray}
\label{eq:loclq}
\nonumber
\left| \int_0^t e^{(t-\tau) (\gamma_1+i\gamma_2)H} f(\tau) d\tau\right|_{L^q} &
=& \left| \int_0^t (-H)^{\frac s2} e^{(t-\tau) (\gamma_1+i\gamma_2)H} f(\tau) d\tau \right|_{{\mathcal{W}}^{-s,q}} \\
& \le & C_1 T^{\delta} \left| (-H)^{\frac s2}f \right|_{L^{\frac{r}{l+1}} (0,T;{\mathcal{W}}^{-\alpha,q})}  \\
& \le & C_1 T^{\delta} |f|_{L^{\frac{r}{l+1}} (0,T;{\mathcal{W}}^{-l(\frac 2p -\beta),q})}. \nonumber
\end{eqnarray}

It remains to estimate the right hand side above taking $f= :|u+Z_{\infty}^{\gamma_1, \gamma_2}|^2 (u+Z_{\infty}^{\gamma_1, \gamma_2}):
= F_0 +F_1+F_2+F_3$ as in \eqref{def:F}. Applying Lemma \ref{lem:bilinear} with $m=2$, and $s=\frac 2p -\beta$ (it is easily checked that the 
conditions are satisfied under the assumptions of Proposition \ref{prop:LpLWP}), we have
$$
|F_0|_{{\mathcal{W}}^{-3(\frac 2p-\beta),q} }\le C |u|_{{\mathcal{W}}^{-s,q}} |u|^2_{{\mathcal{W}}^{\beta,p}},
$$
then H\"older's inequality implies
\begin{equation}
\label{eq:F0}
|F_0|_{L^{\frac r4}(0,T;{\mathcal{W}}^{-3(\frac 2p -\beta),q})} \le CT^{\frac 2r} |u|_{L^\infty(0,T;{\mathcal{W}}^{-s,q})} |u|^2_{L^r(0,T;{\mathcal{W}}^{\beta,p})}.
\end{equation}
In the same way, applying, for $k=1,2,3$ Lemma \ref{lem:bilinear} with $m=3-k$ leads to
\begin{eqnarray*}
& & |F_k|_{L^{\frac{r}{m+2}}(0,T;{\mathcal{W}}^{-(m+1)(\frac 2p-\beta),q})}\\
& & \le CT^\frac 2r \sum_{0\le j\le k} |:Z_{R,\infty}^j Z_{I,\infty}^{k-j}:|_{{\mathcal{W}}^{-s,q}} |u|_{L^r(0,T;{\mathcal{W}}^{\beta,p})}^{3-k}.
\end{eqnarray*}
Combining the above estimate with \eqref{eq:F0} and \eqref{eq:loclq}, we deduce that the second term in the right hand side
of equation \eqref{eq:umild} is in $L^{\infty}(0,T;L^q(\R^2))$ for any $T<T_0^*$. Again, the continuity in time with values in
$L^q(\R^2)$ follows from the strong continuity of the semi-group and the above estimate, and this concludes the proof.
\hfill\qed
\vspace{3mm}


\noindent
{\em Proof of Proposition \ref{thm:Lpbound}.}  
Here, we assume that $\gamma_1$, $\gamma_2$ and $q$ satisfy the assumptions of Proposition \ref{thm:Lpbound}, that is the dissipation is sufficiently large.
All the  formal computations below to estimate the $L^q$ norm of the solution may be justified as in \cite{mw} (see also Remark \ref{rem:mw} below).
Let $u$ be solution of 
\begin{equation}
\label{eq:lq}
 \partial_t u -(\gamma_1+i\gamma_2) Hu= -(\gamma_1+i\gamma_2) |u|^2 u +\Theta(u,Z),  \quad t \in [0, T_1^* \wedge T],
 \end{equation}
where we have denoted simply $Z=Z_{\infty}^{\gamma_1, \gamma_2}$ and  
$$\Theta(u,Z)= -\left(\gamma_1+i\gamma_2) (F_1(u,Z)+F_2(u,Z)+F_3(u,Z)\right),$$ with $F_j$, $j=1,2,3$ defined as in \eqref{def:F}.
Multiplying equation \eqref{eq:lq} by $|u|^{q-2}\bar u$ and integrating the real part of the result over $\R^2$, we obtain:
\begin{eqnarray*}
\frac{1}{q} \frac{d}{dt} |u|_{L^q}^q &=& ((\gamma_1+i\gamma_2) Hu- (\gamma_1+\gamma_2) |u|^2u+\Theta(u,Z), |u|^{q-2}u)_{L^2} \\
&=& ((\gamma_1+i\gamma_2) \Delta u, |u|^{q-2}u)_{L^2}-\gamma_1( |x|^2 u, |u|^{q-2}u)_{L^2} \\
&&-\gamma_1 ( |u|^2 u, |u|^{q-2}u)_{L^2} -((\gamma_1+i\gamma_2)(F_1+F_2+F_3), |u|^{q-2}u)_{L^2}. 
\end{eqnarray*}
The first term in the right hand side can be written by integration by parts as follows.
\begin{eqnarray} \nonumber 
 ((\gamma_1+i\gamma_2) \Delta u, |u|^{q-2}u)_{L^2} &=&-((\gamma_1+i\gamma_2) \nabla u, \nabla(|u|^{q-2}u))_{L^2} \\ \nonumber 
&=& -\gamma_1 \int_{\R^2} |u|^{q-2} |\nabla u|^2 dx  
 -\gamma_1 (q-2) \int_{\R^2} (\re(\bar{u} \nabla u))^2 |u|^{q-4} dx \\ 
& &+\gamma_2 (q-2) \int_{\R^2} \im (\bar{u} \nabla u) \re (\bar{u} \nabla u) |u|^{q-4} dx.\label{est:1}
\end{eqnarray}
Writing
$$
|u|^{q-2}|\nabla u|^2 =|u|^{q-4} [\re(\bar u\nabla u)^2+\im(\bar u\nabla u)^2],
$$
the above term may be bounded, for any $\delta \in (0,1)$, by
\begin{eqnarray*} \nonumber 
& & -\gamma_1 \delta \int_{\R^2} |\nabla u|^2 |u|^{q-2} dx 
-\gamma_1(q-1-\delta)\int_{\R^2} (\re (\bar{u} \nabla u))^2 |u|^{q-4} dx\\ 
& &- \gamma_1(1-\delta) \int_{\R^2} (\im (\bar{u} \nabla u))^2 |u|^{q-4} dx
+\gamma_2 (q-2) \int_{\R^2}  \im (\bar{u} \nabla u) \re (\bar{u} \nabla u) |u|^{q-4} dx. \label{estim2}
\end{eqnarray*}
Let, for $\delta \in (0,1)$,
$$ 
A_\delta:=
\begin{pmatrix}
\gamma_1 (q-1-\delta) & -\gamma_2 \left(\frac{q}{2}-1 \right) \\
-\gamma_2 \left(\frac{q}{2}-1 \right) & \gamma_1(1-\delta)
\end{pmatrix},
$$
then $\mathrm{tr} A_\delta \ge 0$ since $q \ge 2$. On the other hand, 
$\mathrm{det} A_\delta \ge 0$ if and only if : 
either $\gamma_2=0$, or 
$\gamma_2 \ne 0$ and $q-2 \le 2(1-\delta) \{\kappa^2 +\kappa \sqrt{1+\kappa^2}\}$ with $\kappa=\gamma_1/\gamma_2$.
Thus, choosing from now on $\delta$ such that the previous condition is satisfied (which is always possible under the assumptions 
of Theorem \ref{thm:Lpbound}), we obtain
\begin{eqnarray} \nonumber
&&\frac{1}{q} \frac{d}{dt} |u|_{L^q}^q + \gamma_1 \delta  \int_{\R^2} |u|^{q-2} |\nabla u|^2 dx 
+\gamma_1 \int_{\R^2} |x|^2 |u|^{q} dx + \gamma_1 \int_{\R^2} |u|^{q+2} dx \\ \label{est:Lp}
&& \hspace{50mm} \le -((\gamma_1+i\gamma_2)(F_1+F_2+F_3), |u|^{q-2}u)_{L^2}, \\ \nonumber 
\end{eqnarray}
where we recall that
\begin{eqnarray*}
F_1&=& Z |u|^2 +2 Z_{R} u_R u+2 Z_{I} u_I u,\\ 
F_2&=& : Z_{R}^2: (3u_R+iu_I) \, + :Z_{I}^2: (u_R + 3iu_I)
+\, 2:Z_{R}Z_{I}:(u_I+iu_R), \\
F_3 &=& :Z_{R}^3: +\, i:Z_{I}^3: + :Z_{R} Z_{I}^2: + \,i:Z_{R}^2 Z_{I}:.
\end{eqnarray*}
In what follows, we will use the space 
$$\widetilde{\mathcal{W}}^{s,r}(\R^2)=\{f\in \mathcal{S}';~ \langle D \rangle^s f \in L^r(\R^2),~ \langle x \rangle^s f \in L^r(\R^2)\},$$ 
which is equivalent to ${\mathcal{W}}^{s,r}(\R^2)$ for any $r \in (1, \infty)$ by (1) of Proposition \ref{prop:dg}. 
We begin with an estimate for the term containing $F_1$ in the right hand side of \eqref{est:Lp}. We take $\alpha >0$ and $4<p_0<+\infty$ satisfying $\alpha p_0 >8$ and $\frac{1}{p_0^*} \ge 1-\frac{s-\alpha}{2}$. 
Then by duality, (1) of Proposition \ref{prop:dg}, and Sobolev embedding, we have  
\begin{eqnarray*}
|(F_1, |u|^{q-2}u)_{L^2}|&\lesssim& (||u|^q u_R|_{\mathcal{W}^{\alpha, p_0^*}}+||u|^q u_I|_{\mathcal{W}^{\alpha, p_0^*}}) 
|Z|_{\mathcal{W}^{-\alpha,p_0}}\\
&\lesssim&  (||u|^q u_R|_{\widetilde{\mathcal{W}}^{s, 1}}+||u|^q u_I|_{\widetilde{\mathcal{W}}^{s, 1}})
|Z|_{\mathcal{W}^{-\alpha,p_0}}.
\end{eqnarray*}
Using the interpolation 
\begin{equation} \label{ineq:interpol1}
|f|_{\widetilde{\mathcal{W}}^{s,r}} \le |f|_{\widetilde{\mathcal{W}}^{0,r}}^{1-s}|f|_{\widetilde{\mathcal{W}}^{1,r}}^s, \quad s \in [0,1], \quad r \in [1, \infty), 
\end{equation}
the right hand side of the above inequality may be bounded by
\begin{eqnarray*}
||u|^{q+1}|_{L^1}^{1-s} \big(||u|^{q+1}|_{L^1}+ |\nabla(|u|^{q}u)|_{L^1} + |x |u|^{q+1}|_{L^1}\big)^s |Z|_{\mathcal{W}^{-\alpha,p_0}},
\end{eqnarray*} 
and then, by Young inequality, we obtain for any $\varepsilon>0$, 
\begin{eqnarray} \label{ineq:b}
|(F_1, |u|^{q-2}u)_{L^2}|& \le & \varepsilon\big(|u|^{q+1}_{L^{q+1}}+ |\nabla (|u|^{q} u)|_{L^1} + | x |u|^{q+1}|_{L^1}\big) + \varepsilon^{-\frac{s}{1-s}} |Z|_{\mathcal{W}^{-\alpha,p_0}}^{\frac{1}{1-s}} |u|^{q+1}_{L^{q+1}}.
\end{eqnarray}
In order to absorb the terms involving $|u|^{q+1}_{L^{q+1}}$ into the left hand side of \eqref{est:Lp},
we fix $r>0$, and split the integral to get, thanks to H\"older Inequality,
\begin{eqnarray} \nonumber 
|u|^{q+1}_{L^{q+1}} &\le& 
\int_{|x| \le r} |u|^{q+1} + \int_{|x|\ge r} |x||u|^{q/2} |x|^{-1} |u|^{(q+2)/2} dx \\ \label{ineq1:q+1}
&\le& (\pi r^2)^{\frac{1}{q+2}} \left[\int_{\R^2} |u|^{q+2}dx\right]^{\frac{q+1}{q+2}} + r^{-1} \Big[\int_{\R^2} |x|^2 |u|^{q} dx \Big]^{1/2} \Big[ \int_{\R^2} |u|^{q+2} dx \Big]^{1/2}.
\end{eqnarray}
Taking $r=1$, and using Young Inequality we obtain
\begin{equation} \label{ineq2:q+1}
|u|_{L^{q+1}}^{q+1} \le \frac{3}{2} \int_{\R^2}  |u|^{q+2}dx +\frac{1}{2} \int_{\R^2} |x|^2 |u|^q dx+\pi.
\end{equation}
Noting that $\nabla (|u|^{q} u)= \frac{q}{2}  u^2 |u|^{q-2} \nabla \bar u + \frac{q+2}{2} |u|^{q} \nabla u$, 
the second term in (\ref{ineq:b}) can be estimated thanks to Cauchy-Schwarz inequality:
\begin{eqnarray} \label{ineq1:nablaq+1}
|\nabla (|u|^{q} u)|_{L^1} 
\le (q+1) ||u|^{q-2} |\nabla u|^2|_{L^1}^{\frac{1}{2}} ||u|^{q+2}|_{L^1}^{\frac{1}{2}}
\le (q+1)[ ||u|^{q-2} |\nabla u|^2|_{L^1}+|u|^{q+2}_{L^{q+2}}].
\end{eqnarray}
As for the third term in (\ref{ineq:b}), again, we split the integral and use H\"older and Young Inequalities:
\begin{eqnarray}
\label{ineq:xqplus1} \nonumber  
|x|u|^{q+1}|_{L^1} &\le& 
\pi^{\frac{1}{q+2}}  \left[\int_{|x| \le 1} |u|^{q+2}dx\right]^{\frac{q+1}{q+2}} + \int_{|x|\ge 1} |x||u|^{q/2} |u|^{(q+2)/2} dx \\ \label{ineq:xq}
&\le& \pi+ \frac{3}{2}\int_{\R^2}  |u|^{q+2}dx + \frac{1}{2} \int_{\R^2} |x|^2 |u|^qdx.
\end{eqnarray}
Let us finally estimate the last term in \eqref{ineq:b}.
We use (\ref{ineq1:q+1}) with $ r=\frac 12 \varepsilon^{-\frac{1}{1-s}} |Z|_{\mathcal{W}^{-\alpha,p_0}}^{\frac{1}{1-s}}$, and Young Inequality 
to obtain
\begin{eqnarray*}
\varepsilon^{-\frac{s}{1-s}} |Z|_{\mathcal{W}^{-\alpha,p_0}}^{\frac{1}{1-s}} |u|^{q+1}_{L^{q+1}} 
& \le &\Big( \frac{\pi}{4}\Big)^{\frac{1}{q+2}} \varepsilon^{-\frac{1}{1-s}(s+\frac{2}{q+2})} |Z|_{\mathcal{W}^{-\alpha,p_0}}^{\frac{1}{1-s}(1+\frac{2}{q+2})} 
\left[ \int_{\R^2}|u|^{q+2}dx\right]^{\frac{q+1}{q+2}}\\
&&
+ \varepsilon \Big[\int_{\R^2} |x|^2 |u|^{q} dx + \int_{\R^2} |u|^{q+2} dx \Big]\\
& \le & \Big(\frac{\pi}{4}\Big)^{\frac{1}{q+2}} \varepsilon^{-\frac{1}{1-s}(s+\frac{2}{q+2})} |Z|_{\mathcal{W}^{-\alpha,p_0}}^{\frac{1}{1-s}(1+\frac{2}{q+2})} 
\left[ \varepsilon_1 \int_{\R^2}|u|^{q+2} dx+ \varepsilon_1^{-(q+1)} \right]\\
&&
+ \varepsilon \Big[\int_{\R^2} |x|^2 |u|^{q} dx + \int_{\R^2} |u|^{q+2} dx \Big],
\end{eqnarray*}
for any $\varepsilon_1>0$.
Choosing then
$$
\varepsilon_1= \Big( \frac{4}{\pi} \Big)^{\frac{1}{q+2}} \varepsilon^{\frac{1}{1-s}(1+\frac{2}{q+2})} 
|Z|_{\mathcal{W}^{-\alpha,p_0}}^{-\frac{1}{1-s}(1+\frac{2}{q+2})},
$$
we obtain
\begin{equation}
\label{lastterm}
 \varepsilon^{-\frac{s}{1-s}} |Z|_{\mathcal{W}^{-\alpha,p_0}}^{\frac{1}{1-s}} |u|^{q+1}_{L^{q+1}} 
 \le  \varepsilon \Big[\int_{\R^2} |x|^2 |u|^{q} dx + 2\int_{\R^2} |u|^{q+2} dx \Big]
+ \frac{\pi}{4} \varepsilon^{-\frac{s(q+2)+q^2+5q+6}{(q+2)(1-s)}} |Z|_{\mathcal{W}^{-\alpha,p_0}}^{\frac{q+4}{1-s}}.
\end{equation}
Gathering \eqref{ineq2:q+1}, \eqref{ineq1:nablaq+1}, \eqref{ineq:xq} and the above inequality, the right hand side of \eqref{ineq:b} is bounded as follows:
\begin{eqnarray}
\label{estimF1} \nonumber
|(F_1, |u|^{q-2}u)_{L^2}| &\le&2\eps q \left[\int_{\R^2} |u|^{q-2}|\nabla u|^2 dx+\int_{\R^2} |x|^2 |u|^{q} dx
+ \int_{\R^2} |u|^{q+2} dx\right]\\
& & + \, C_{\eps,s,q} (1+ |Z|_{\mathcal{W}^{-\alpha,p_0}}^{\frac{q+4}{1-s}}),
\end{eqnarray}
for some constant $C_{\eps,s,q}$.

Next, we estimate $(F_3, |u|^{q-2}u)_{L^2}$.  For $s\in (0,1]$ with $sq>8$ and $2<q <+\infty$, 
$$ 
| (F_3, |u|^{q-2}u)_{L^2}| = |(:|Z|^2Z:,|u|^{q-2}u)| \le |F_3|_{\mathcal{W}^{-s,q}} ||u|^{q-2} u|_{\mathcal{W}^{s,q^*}} .
$$     
Similarly to the estimate for $(F_1, |u|^{q-2}u)_{L^2}$, by interpolation (\ref{ineq:interpol1}), (1) of Proposition \ref{prop:dg} and Young inequality, we obtain for any $\varepsilon>0$,
\begin{eqnarray}
\label{ineq:a}
\nonumber
| (F_3, |u|^{q-2}u)_{L^2}| &\le & \varepsilon (|u|^{q-1}_{L^{q}} +|\nabla(|u|^{q-2} u)|_{L^{q^*}}+ |x |u|^{q-1}|_{L^{q^*}})\\
& & +\, \varepsilon^{-\frac{s}{1-s}} |F_3|_{\mathcal{W}^{-s,q}}^{\frac{1}{1-s}}~ |u|^{q-1}_{L^{q}}. 
\end{eqnarray}
Since $\frac{1}{q^*}=\frac 12 + \frac{q-2}{2q}$, it follows by H\"older, Cauchy-Schwarz and Young inequalities that 
\begin{eqnarray}
\label{ineq:grad}
\nonumber
|\nabla (|u|^{q-2} u)|_{L^{q^*}} &\le &(q-1) ||u|^{q-2} \nabla u|_{L^{q^*}} 
\le (q-1) ||u|^{\frac{q-2}{2}} \nabla u|_{L^2} ||u|^{\frac{q-2}{2}}|_{L^{\frac{2q}{q-2}}} \\
& & \le (q-1) \left[  ||u|^{\frac{q-2}{2}} \nabla u|_{L^2}^2 + |u|_{L^q}^q +1\right].
\end{eqnarray}
Moreover, splitting the integrals as above, we obtain on the one hand
\begin{eqnarray} \label{ineq2:q}
|u|^{q-1}_{L^{q}}  \le |u|_{L^q}^q +1 \le 
\frac 32 \int_{\R^2} |u|^{q+2} + \int_{\R^2} |x|^2 |u|^q +2 \pi,
\end{eqnarray}
and on the other hand,
\begin{eqnarray}
\label{ineq:x} \nonumber
|x|u|^{q-1}|_{L^{q^*}} 
&\le& \int_{|x|\le 1} |x|^{\frac{q}{q-1}} |u|^q + \int_{|x| \ge 1} |x|^2 |u|^{q} |x|^{-\frac{q-2}{q-1}} +1\\ 
&\le & \int_{\R^2} |u|^{q+2} +  \int_{\R^2} |x|^2 |u|^q + \, 2\pi. 
\end{eqnarray}
Let us estimate the last term in (\ref{ineq:a}). We proceed as in \eqref{lastterm} : choosing $r>0$ such that
$
\eps^{-\frac{s}{1-s}} |F_3|_{\mathcal{W}^{-s,q}}^{\frac{1}{1-s}} r^{-2} =\eps,
$
then by the inequality
$$
|u|_{L^q}^q \le  
(\pi r^2)^{\frac{2}{q+2}} \left[ \int_{|x|\le r} |u|^{q+2} \right]^{\frac{q}{q+2}} +r^{-2} \int_{|x|\ge r} |x|^2 |u|^q,
 $$
we have
\begin{eqnarray*}
\eps^{-\frac{s}{1-s}} |F_3|_{\mathcal{W}^{-s,q}}^{\frac{1}{1-s}}~ |u|^{q-1}_{L^{q}} 
&\le& \eps^{-\frac{s}{1-s}} |F_3|_{\mathcal{W}^{-s,q}}^{\frac{1}{1-s}} (|u|_{L^q}^q +1) \\ 
&\le &\eps^{-\frac{s}{1-s}} |F_3|_{\mathcal{W}^{-s,q}}^{\frac{1}{1-s}}
\left[1+ (\pi r^2)^{\frac{2}{q+2}} \left[ \int_{\R^2} |u|^{q+2} \right]^{\frac{q}{q+2}}\right] + \eps \int_{\R^2} |x|^2 |u|^q.
\end{eqnarray*}
Applying then the Young Inequality
$$\left[ \int_{\R^2} |u|^{q+2} \right]^{\frac{q}{q+2}} \le \varepsilon_2 \int_{\R^2} |u|^{q+2} + \varepsilon_2^{-\frac{q}{2}}$$
with a well chosen $\eps_2$, results in the following bound, for any  $\varepsilon >0$ :
\begin{eqnarray}
\label{ineq:last}
\eps^{-\frac{s}{1-s}} |F_3|_{\mathcal{W}^{-s,q}}^{\frac{1}{1-s}}~ |u|^{q-1}_{L^{q}}\nonumber
& \le & \eps^{-\frac{s}{1-s}} |F_3|_{\mathcal{W}^{-s,q}}^{\frac{1}{1-s}} 
+C_{\eps,s,q} |F_3|_{\mathcal{W}^{-s,q}}^{\frac{1}{1-s}(2+\frac{q}{2})} \\
& & +\, \eps \left[\int_{\R^2} |u|^{q+2} + \int_{\R^2} |x|^2 |u|^q \right]. 
\end{eqnarray}
Plugging \eqref{ineq:grad}, \eqref{ineq2:q}, \eqref{ineq:x} and \eqref{ineq:last} into \eqref{ineq:a}, we thus may bound the right hand side of this later inequality as follows:
\begin{eqnarray}
\label{estim:a}
\nonumber
| (F_3, |u|^{q-2}u)_{L^2}| & \le &2 \eps q \left[\big| |u|^{q-2} |\nabla u|^2 \big|_{L^1} + |u|_{L^{q+2}}^{q+2} +\int_{\R^2} |x|^2 |u|^q\right] \\
& & + \, \eps^{-\frac{s}{1-s}} |F_3|_{\mathcal{W}^{-s,q}}^{\frac{1}{1-s}} +C_{\eps,s,q} |F_3|_{\mathcal{W}^{-s,q}}^{\frac{q+4}{2(1-s)}}.
\end{eqnarray}

We use again, in order to estimate $(F_2, |u|^{q-2}u)_{L^2}$, an auxiliary pair $\alpha>0$ and $p_0 \in [4, \infty)$ such that $\alpha p_0>8$ and $\frac{1}{p_0^*} \ge 1-\frac{s-\alpha}{2}$.
By the same arguments as above, for any $\varepsilon>0$, we obtain as in \eqref{ineq:b}-\eqref{ineq:xqplus1},
\begin{eqnarray*}
|(F_2, |u|^{q-2}u)_{L^2}| &\lesssim& ||u|^q|_{\widetilde{\mathcal{W}}^{s,1}} |:Z^2:|_{\mathcal{W}^{-\alpha, p_0}} \\
 &\le & \varepsilon \left[(q+2) \int_{\R^2} |u|^{q+2} +q \int_{\R^2} |u|^{q-2} |\nabla u|^2 
 +(q+3) \int |x|^2 |u|^q +(q+2) \pi \right]   \\
&& + \varepsilon^{-\frac{s}{1-s}} |:Z^2:|_{\mathcal{W}^{-\alpha,p_0}}^{\frac{1}{1-s}} |u|^q_{L^q}.
 \end{eqnarray*}
Estimating the last term in the right hand side of the above inequality as in \eqref{ineq:last}, we finally get, using Young's inequality once more,
\begin{eqnarray}\label{estim:F2} \nonumber
|(F_2, |u|^{q-2}u)_{L^2}|&\le& 2\eps q \left[\int_{\R^2} |u|^2 |\nabla u|^2 +\int_{\R^2} |x|^2 |u|^q +\int_{\R^2} |u|^{q+2} \right]\\
& & + \,C_{\eps, s,q} (1+ |:Z^2:|_{\mathcal{W}^{-\alpha,p_0}}^{\frac{q+4}{2(1-s)}}).
\end{eqnarray}
Gathering the estimates \eqref{estimF1}, \eqref{estim:a} and \eqref{estim:F2} and plugging them into \eqref{est:Lp}, choosing $\eps>0$
small enough so that $6\eps q(|\gamma_1|+|\gamma_2|) \le \frac{\gamma_1 \delta}{2}$, with $\delta=1$ if $\gamma_2=0$, and 
$\delta=1-\frac{q-2}{2(\kappa^2+\kappa\sqrt{1+\kappa^2})}$ if $\gamma_2 \not =0$, finally leads to
\begin{eqnarray*} 
&&\frac{1}{q} \frac{d}{dt} |u|_{L^q}^q + \frac 12 \gamma_1 \delta \int_{\R^2} |u|^{q-2} |\nabla u|^2 dx 
+\frac{\gamma_1}{2}  \int_{\R^2} |x|^2 |u|^{q} dx + \frac{\gamma_1}{2} \int_{\R^2} |u|^{q+2} dx \\
& & \le C_{\gamma_1,\gamma_2,s,q} \left[ 1+ |Z|_{\mathcal{W}^{-\alpha,p_0}}^{\frac{q+4}{1-s}} + |:Z^2:|_{\mathcal{W}^{-\alpha,p_0}}^{\frac{q+4}{2(1-s)}}
+|:|Z|^2Z:|_{\mathcal{W}^{-s,q}}^{\frac{q+4}{2(1-s)}} \right].
\end{eqnarray*}
At last, using \eqref{ineq2:q} again, the following  more precise inequality holds :
\begin{eqnarray*} 
&&\frac{1}{q} \frac{d}{dt} |u|_{L^q}^q  +\frac{\gamma_1}{8} |u|^q_{L^q} +\frac 12 \gamma_1 \delta \int_{\R^2} |u|^{q-2} |\nabla u|^2 dx 
+\frac{\gamma_1}{4}  \int_{\R^2} |x|^2 |u|^{q} dx + \frac{\gamma_1}{4} \int_{\R^2} |u|^{q+2} dx \\
& & \le C_{\gamma_1,\gamma_2,s,q} \left[ 1+ |Z|_{\mathcal{W}^{-\alpha,p_0}}^{\frac{q+4}{1-s}} + |:Z^2:|_{\mathcal{W}^{-\alpha,p_0}}^{\frac{q+4}{2(1-s)}}
+|:|Z|^2Z:|_{\mathcal{W}^{-s,q}}^{\frac{q+4}{2(1-s)}} \right]
\end{eqnarray*}
from which we deduce
\begin{eqnarray} 
\label{estim:final}
& & |u(t)|_{L^q}^q \le e^{-\frac{\gamma_1q t}{8}}|u_0|_{L^q}^q \\
&  &\hskip 0.05 in +\, C_{\gamma_1,\gamma_2,s,q} \int_0^t e^{-\frac{\gamma_1q}{8} (t-\sigma)}
\left[ 1+ |Z|_{\mathcal{W}^{-\alpha,p_0}}^{\frac{q+4}{1-s}} + |:Z^2:|_{\mathcal{W}^{-\alpha,p_0}}^{\frac{q+4}{2(1-s)}}
+|:|Z|^2Z:|_{\mathcal{W}^{-s,q}}^{\frac{q+4}{2(1-s)}} \right].\nonumber
\end{eqnarray}
The result follows.
\hfill\qed
\vspace{3mm}

\begin{rem} \label{rem:mw}
These computations may be justified as in \cite{mw}. Indeed, it is not difficult to prove, using similar estimates as for the proof of Proposition \ref{prop:LpLWP},
that the solution $u$ is in $C^{1/2+}(0,\tau;L^{q}(\R^2))$, for any $\tau<T^*$, a.s., provided the initial state $u_0$ is sufficiently regular. Now, the local solution 
of Theorem \ref{thm:localexistence} being obtained thanks to a fixed point argument, it is continuous with respect to the initial state, so that proving the estimate 
for regular initial data is
actually sufficient.
\end{rem}
\vspace{3mm}

Finally we show the global existence result.
\vspace{3mm}

\noindent
{\em Proof of Theorem \ref{thm:global}.} 
Here again, we assume that $\gamma_1$, $\gamma_2$ and $q$ satisfy the assumptions of Proposition \ref{thm:Lpbound}, that is the dissipation is sufficiently large.
Let $u$ be the solution of \eqref{eq:umild} given by Theorem \ref{thm:localexistence}, and let
 $0<\tilde{T_0} <1\wedge T_0^* \wedge T$, then
\begin{eqnarray*}
u(\tilde{T_0}) &=& e^{\tilde{T_0} (\gamma_1+i\gamma_2) H} u_0 -(\gamma_1+i\gamma_2) \int_0^{\tilde{T_0}} e^{(\tilde{T_0}-\tau) (\gamma_1+i\gamma_2) H} :|u+Z_{\infty}^{\gamma_1, \gamma_2}|^2(u+Z_{\infty}^{\gamma_1, \gamma_2}):(\tau) d\tau. 
\end{eqnarray*}
Thus,  using Lemma \ref{lem:regularization}  and the same estimates as in the proof of Proposition \ref{prop:LpLWP}, it is not difficult to see that $u(\tilde T_0) \in L^q(\R^2)$;
we may then solve equation (\ref{eq:u_wick}) starting from $\tilde{T_0}$ using Proposition \ref{prop:LpLWP}, and the combination of Theorem \ref{thm:localexistence},
Proposition \ref{prop:LpLWP} and Proposition \ref{thm:Lpbound} shows that the solution is global in time, concluding the statement of Theorem \ref{thm:global}.
\hfill\qed

\subsection{Galerkin approximation and tightness of the Gibbs measure}
In order to get an invariant measure for equation \ref{eq:SGPE_Math}, we will use Galerkin approximations, and show that their family of Gibbs measure 
is tight in a certain function space. 

Let us  consider the Galerkin approximation to (\ref{eq:SGL}):
\begin{equation} \label{eq:SGL_finite} 
dX=(\gamma_1 +i\gamma_2)\big(HX-S_N(:|S_N X|^2 S_N X:)\big)dt +\sqrt{2\gamma_1} \Pi_N dW, \quad X(0)=X_0 \in E_N^{\C}.
\end{equation}
This finite dimensional system has a unique invariant measure of the form:
$$d \tilde{\rho}_N (y)= \Gamma_N e^{-\tilde \cH_{N}(S_N y)} dy, \quad y \in E_N^{\C},$$
where $\Gamma_N^{-1}= \int e^{-\tilde \cH_{N}(S_N y)} dy $, and  
$$\tilde \cH_{N}(y)=\frac{1}{2}|\nabla y|_{L^2}^2 +\frac{1}{2}|x y|_{L^2}^2 + \int_{\R^2}\left[\frac{1}{4}|y(x)|^4 -2\rho_N^2(x)|y(x)|^2+2 \rho_N^4(x)\right]dx,$$
$\rho_N$ being the normalization function defined in \eqref{def:rhoN}. Note that
$$\tilde \cH_{N}(y)=\frac{1}{2}|\nabla y|_{L^2}^2 +\frac{1}{2}|x y|_{L^2}^2 + \frac 14 \int_{\R^2} :|y(x)|^4: dx, $$
and $\nabla_y \tilde \cH_N(y)=-Hy+:|y|^2y:$. 
The shifted equation associated with \eqref{eq:SGL_finite} is
\begin{equation} \label{eq:u_finite} 
\frac{du}{dt}=(\gamma_1+i\gamma_2)\left[ Hu-S_N\big(:|S_N( u+ Z_{\infty, N}^{\gamma_1, \gamma_2})|^2 S_N (u+ Z_{\infty,N}^{\gamma_1, \gamma_2}) :\big) \right],   
\end{equation}
where 
$$ {Z}_{\infty, N}^{\gamma_1, \gamma_2}(t)=\sqrt{2\gamma_1} \int_{-\infty}^t e^{(t-\tau)(\gamma_1+i\gamma_2)H} \Pi_N dW(\tau).$$
We may of course apply Theorem 1 to equation \eqref{eq:u_finite} and prove, using the same arguments as in \cite{dbdf}, Proposition 3, 
that the solution $u_N$ is globally defined.  We deduce the following result for equation \eqref{eq:SGL_finite} by Proposition \ref{prop:ReImCauchy}:
\begin{prop} \label{prop:convGalerkin} Fix any $T>0$. Let $\gamma_1>0$ and $q>p>3r$, $r>6$. Assume 
$0<s<\beta<2/p$, $qs>8$, 
$ \big(\beta-s\big)>\big(\frac{2}{p}-\beta\big)$,  and $s+2 \big(\frac{2}{p} -\beta\big) < 2\big(1-\frac{1}{q}\big)$.
Then there exists a unique global solution in $C([0,T],E_N^{\C})$, denoted by $X_N$, of (\ref{eq:SGL_finite}). 
Moreover, if $X_N(0)$ converges to $X_0$ a.s. in $\mathcal{W}^{-s,q}(\R^2)$, then
$X_N$ converges to $X$ in $C([0,T];\mathcal{W}^{-s,q}(\R^2))$ a.s. as $N$ goes to infinity, for any $T<T_0^*$, where $X=u+Z_\infty^{\gamma_1,\gamma_2}$,
$u$ being the solution of equation \eqref{eq:u_wick} given by Theorem \ref{thm:localexistence} and $T_0^*$ being its maximal existence time.
\end{prop}
\vspace{3mm}

Define the Feller transition semigroup 
$P_{t}^N \phi (y) = \E(\phi(X_N(t,y)))$, for $y\in E_N^{\C}$, associated with equation \eqref{eq:SGL_finite}.
The finite-dimensional measure $\tilde{\rho}_N$ is an invariant measure for $(P_t^N)_{t\ge 0}$,     
according to the same argument as in Proposition 4 of \cite{dbdf}. We now wish to prove the tightness of
the family of measures $(\tilde \rho_N)_N$. Note that we cannot directly use the $L^q$-bound provided by 
Proposition \ref{thm:Lpbound} due to the presence of the cut-off $S_N$ in front of the nonlinear term in
\eqref{eq:u_finite}.

Alternatively, considering the coupled evolution on $E_N^\C$ given by
\begin{equation}
\label{eq:coupled}
\left\{ \begin{array}{rcl} \ds
\frac{du}{dt}& =&(\gamma_1+i\gamma_2)\left[ Hu-S_N\big(:|S_N(u+Z)|^2 S_N (u+Z):\big)\right] \\[0.25cm]
dZ &= & (\gamma_1+i\gamma_2) HZ dt +\sqrt{2\gamma_1}\Pi_N dW,
\end{array} \right.
\end{equation}
one may easily prove, using e.g. similar estimates as in the proof of Proposition \ref{prop:tightness} below, together with the Gaussianity
of $Z$, and a Krylov-Bogolyubov argument, that \eqref{eq:coupled} has an invariant measure $\nu_N$ on $E_N^\C \times \E_N^\C$.
Moreover, by uniqueness of the invariant measure of \eqref{eq:SGL_finite}, we necessarily have for any bounded continuous
function $\varphi$ on $\E_N^\C$ :
$$
\int_{E_N^\C} \varphi(x)\tilde \rho_N(dx)=\int \!\!\int_{E_N^\C\times E_N^\C} \varphi(u+z) \nu_N(du,dz).
$$
The next proposition will imply the tightness of the sequence $(\tilde \rho_N)_N$ in $\mathcal{W}^{-s,q}$. 

\begin{prop}
\label{prop:tightness}
Let $(u_N,Z_N) \in C(\R_+;E_N^\C \times E_N^\C)$ be a stationary solution of \eqref{eq:coupled}. Then, for any $m > 0$, there is a constant $C_m>0$ 
independent of $t$ and $N$, such that
\begin{equation}
\label{eq:h1bound}
\E(|(-H)^{\frac{1}{2m}}u_N|_{L^2}^{2m}) \le C_m.
\end{equation}
\end{prop}

\begin{cor}
\label{cor:tightness}
The family of finite dimensional Gibbs measures $(\tilde{\rho_N})_N$ is tight in $\mathcal{W}^{-s,q}$ for any $q>8$ and $s> \frac8q$.
\end{cor}

\noindent
{\em Proof of Proposition \ref{prop:tightness}.}
Taking the $L^2$-inner product of the first equation in \eqref{eq:coupled} with $u_N$ and using \eqref{def:F} yields
\begin{eqnarray}
\label{bound:F} 
\ds \nonumber& & 
\frac12 \frac{d}{dt} |u_N(t)|_{L^2}^2 +\gamma_1 |(-H)^{\frac12} u_N(t)|_{L^2}^2 +\gamma_1 |S_Nu_N(t)|_{L^4}^4  =  
-\mathrm{Re}(\gamma_1+i\gamma_2) \int_{\R^2} \left[ F_1(S_Nu_N, S_NZ_N)\right.\\
& & \hskip 1 in \left. + \,F_2(S_Nu_N, S_NZ_N)+F_3(S_NZ_N)\right]
\overline{S_Nu_N}(t) dx.
\end{eqnarray}
We first estimate the term containing $F_3$ in the right hand side above. Thanks to Proposition~\ref{prop:ReImCauchy},
taking $0<s<1$ and $q>8$ such that $sq>8$, we may bound
$$
\left| \int_{\R^2} F_3(S_NZ_N)\overline{S_Nu_N} dx \right| \lesssim |F_3(S_NZ_N)|_{\mathcal{W}^{-s,q}} |S_Nu_N|_{\mathcal{W}^{s,q'}}
$$
with $\frac1q +\frac{1}{q'} =1$. Interpolating then $\mathcal{W}^{s,q'}$between $L^r$ and $\mathcal{W}^{1,2}$, with $\frac{1}{q'}=\frac s2 +\frac{1-s}{r}$,
we get
$$
|S_Nu_N|_{\mathcal{W}^{s,q'}} \lesssim |(-H)^{\frac12} S_Nu_N|_{L^2}^s |S_Nu_N|_{L^r}^{1-s}.
$$
On the other hand, noticing that $r\in(1,2)$, we have for any $v\in \mathcal{W}^{1,2}$:
\begin{eqnarray*}
\int_{|x|\ge 1} |v(x)|^r dx & \le & \left[\int_{|x|\ge 1}|x|^2 |v(x)|^2 dx\right]^{\frac r2} \left[ \int_{|x|\ge 1} |x|^{-\frac{2r}{2-r}} dx\right]^{\frac{2-r}{2r}}\\
& \lesssim &|(-H)^{\frac12 } v|_{L^2}^r,
\end{eqnarray*}
so that $|v|_{L^r}\lesssim |(-H)^{\frac12 } v|_{L^2}$. It follows that
\begin{eqnarray}
\label{bound:F3}
\nonumber \left| \int_{\R^2} F_3(S_NZ_N) \overline{S_Nu_N}dx\right| 
&  \lesssim &|F_3(S_NZ_N)|_{\mathcal{W}^{-s,q}} |(-H)^{\frac12} S_Nu_N|_{L^2}\\
&  \le & C |F_3(S_NZ_N)|_{\mathcal{W}^{-s,q}}^2 +\frac{\gamma_1}{4(\gamma_1+|\gamma_2|)} |(-H)^{\frac12} S_Nu_N|_{L^2}^2.
\end{eqnarray}

Next, we consider the term containing $F_2$ in the right hand side of \eqref{bound:F}. First, by \eqref{def:F} and Proposition \ref{prop:dg},
\begin{eqnarray*}
 \left| \int_{\R^2} F_2(S_Nu_N, S_NZ_N) \overline{S_Nu_N}dx\right|
&  \lesssim &\sum_{l=0}^2 |:S_NZ_{N,R}^l S_NZ_{N,I}^{2-l} :|_{\mathcal{W}^{-s,q}} |(S_N u_N)^2|_{\mathcal{W}^{s,q'}} \\
&  \lesssim &\sum_{l=0}^2 |:S_NZ_{N,R}^l Z_{N,I}^{2-l} :|_{\mathcal{W}^{-s,q}} |S_Nu_N|_{L^4} |S_Nu_N|_{\mathcal{W}^{s,p}}
\end{eqnarray*}
with $\frac{1}{q'}=\frac14 + \frac1p$. Note that $p\in(1,2)$ and we may use the same procedure as before to obtain
$$
|S_Nu_N|_{\mathcal{W}^{s,p}} \lesssim |(-H)^{\frac12} S_Nu_N|_{L^2},
$$
so that 
\begin{eqnarray}
\label{bound:F2}
& & \nonumber \left| \int_{\R^2} F_2(S_Nu_N, S_NZ_N) \overline{S_Nu_N}\right| \\
&  \le  &C \sum_{l=0}^2 |:S_NZ_{N,R}^l S_NZ_{N,I}^{2-l}:|_{\mathcal{W}^{-s,q}}^4
 +\, \frac{\gamma_1}{2(\gamma_1+|\gamma_2|)} |S_Nu_N|_{L^4}^4 \\
& & \nonumber + \frac{\gamma_1}{4(\gamma_1+|\gamma_2|)} |(-H)^{\frac12}S_Nu_N|_{L^2}^2.
\end{eqnarray}

We finally turn to the term containing $F_1$ in the right hand side of \eqref{bound:F}. 
We easily get, thanks again to Proposition \ref{prop:dg},
\begin{eqnarray*}
 \left| \int_{\R^2} F_1(S_Nu_N, S_NZ_N) \overline{S_Nu_N}dx\right|
&  \lesssim &|S_N Z_N|_{\mathcal{W}^{-s,q}} |S_Nu_N|_{L^4}^2 |S_Nu_N|_{\mathcal{W}^{s,r}}
\end{eqnarray*}
where $r>2$ is such that $\frac{1}{q'}=\frac12 + \frac1r$. 
Let $m>2$ with $\frac1r=\frac s2 +\frac{1-s}{m}$, so that
$$
|S_Nu_N|_{\mathcal{W}^{s,r}}\lesssim |(-H)^{\frac12}S_Nu_N|_{L^2}^s |S_Nu_N|_{L^m}^{1-s}.
$$
If $m> 4$, we interpolate $L^m$ between $L^4$ and $L^{2m}$, then use the Sobolev embedding $\mathcal{W}^{1,2} \subset L^{2m}$.
If $2<m\le4$, we interpolate $L^m$ between $L^2$ and $L^4$, then use $\mathcal{W}^{1,2} \subset L^2$.
In both cases we obtain, using in addition the Poincar\'e inequality for $(-H)$:
\begin{eqnarray*}
|S_Nu_N|_{\mathcal{W}^{s,r}} & \lesssim & |(-H)^{\frac12}S_Nu_N|_{L^2}^\alpha |S_Nu_N|_{L^4}^{1-\alpha}
\end{eqnarray*}
for some constant $\alpha\in (0,1)$. We deduce that 
\begin{eqnarray}
& & \nonumber \left| \int_{\R^2} F_1(S_Nu_N, S_NZ_N) \overline{S_Nu_N}\right| \\ \nonumber
&  \lesssim  & |S_NZ_N|_{\mathcal{W}^{-s,q}} |S_Nu_N|_{L^4}^{3-\alpha} |(-H)^{\frac12} S_Nu_N|_{L^2}^\alpha 
\\ \label{bound:F1}
& \le & C |S_NZ_N|_{\mathcal{W}^{-s,q}}^{\frac{4}{1-\alpha}} + \frac{\gamma_1}{2(\gamma_1+|\gamma_2|)}\left[ |S_Nu_N|_{L^4}^4 
+\frac 12 |(-H)^{\frac12}S_Nu_N|_{L^2}^2\right],
\end{eqnarray}
by Young inequality. 

Gathering \eqref{bound:F}--\eqref{bound:F1}, and noticing that $|(-H)^{\frac12}S_Nu_N|_{L^2}\le |(-H)^{\frac12}u_N|_{L^2}$,
leads to
\begin{eqnarray*}
 \frac{d}{dt} |u_N(t)|_{L^2}^2 +\frac{\gamma_1}{2} |(-H)^{\frac12} u_N(t)|_{L^2}^2
&  \lesssim & \sum_{k+l=1}^3 |:S_NZ_{N,R}^l S_NZ_{N,I}^{k}:|_{\mathcal{W}^{-s,q}}^{m_{k,l}}
\end{eqnarray*} 
for some integers $m_{k,l}$. 
Now, let $m>0.$ We multiply by $|u_N(t)|_{L^2}^{2m-2}$ both sides of the above inequality to get 
\begin{eqnarray*}
 \frac{1}{m} \frac{d}{dt} |u_N(t)|_{L^2}^{2m} +\frac{\gamma_1}{2} |(-H)^{\frac12} u_N(t)|_{L^2}^2 |u_N(t)|_{L^2}^{2m-2}
&  \lesssim & |u_N(t)|_{L^2}^{2m-2} \sum_{k+l=1}^3 |:S_NZ_{N,R}^l S_NZ_{N,I}^{k}:|_{\mathcal{W}^{-s,q}}^{m_{k,l}}.
\end{eqnarray*}
Applying the interpolation inequality
\begin{eqnarray*}
|(-H)^{\frac{1}{2m}} u_N|_{L^2} \le C_m |(-H)^{\frac{1}{2}} u_N|_{L^2}^{\frac{1}{m}} |u_N|_{L^2}^{\frac{m-1}{m}}
\end{eqnarray*}
to the second term in the left hand side, and using Young inequality in the right hand side, we obtain
\begin{eqnarray*}
 \frac{1}{m} \frac{d}{dt} |u_N(t)|_{L^2}^{2m} +\frac{\gamma_1}{2 C_m^{2m}} |(-H)^{\frac{1}{2m}} u_N(t)|_{L^2}^{2m}
&  \lesssim & \varepsilon |u_N(t)|_{L^2}^{2m} + C_{\varepsilon} \sum_{k+l=1}^3 |:S_NZ_{N,R}^l S_NZ_{N,I}^{k}:|_{\mathcal{W}^{-s,q}}^{m'_{k,l}},
\end{eqnarray*}
for any $\varepsilon>0$ and constants $C_{\varepsilon}>0$ and $m'_{k,l}>0$. We choose $\varepsilon=\frac{\gamma_1 \lambda_1^2}{4 C_m^{2m}}$ after  
using Poincar\'{e} inequality so that the first term of the right hand side is absorbed in the left hand side, 
\begin{eqnarray*}
\frac{1}{m} \frac{d}{dt} |u_N(t)|_{L^2}^{2m} +\frac{\gamma_1}{4 C_m^{2m}} |(-H)^{\frac{1}{2m}} u_N(t)|_{L^2}^{2m}
&  \lesssim & \sum_{k+l=1}^3 |:S_NZ_{N,R}^l S_NZ_{N,I}^{k}:|_{\mathcal{W}^{-s,q}}^{m'_{k,l}}.
\end{eqnarray*}
Integrating in time, taking expectations on both sides and using the stationarity of $u_N$ and of the Wick
products, together with Corollary \ref{cor:m-moment}, yields
$$
\E\big(|(-H)^{\frac{1}{2m}}u_N|_{L^2}^{2m}\big)\lesssim \sum_{k+l=1}^3 M_{s,q,k,l,m'_{k,l}},
$$ 
and the conclusion.
\hfill
\qed

\medskip

\noindent
{\em Proof of Corollary \ref{cor:tightness}.}
Let $s'$ with $0<s'<1$ and $q>8$ such that $s'q>8$. Let $(u_N,Z_N)$ be a stationary solution of \eqref{eq:coupled} 
in $E_N^\C\times E_N^\C$. Applying Proposition \ref{prop:tightness} with $m=1$, we deduce that for some positive constant $C$ not depending on $N$, and for any $t\ge 0$,
$$
\E \big( |u_N(t)|_{\mathcal{W}^{-s',q}}^2 \big) \lesssim \E \big(|u_N(t)|_{L^q}^2\big) \lesssim \E\big(|(-H)^{\frac12}u_N|_{L^2}^2\big) \lesssim C,
$$
where we have used the embedding $\mathcal{W}^{1,2} \subset L^q$, for any $q<+\infty$. Thus,
\begin{eqnarray*}
& & \int_{\mathcal{W}^{-s',q}} |x|^2_{\mathcal{W}^{-s',q}} \tilde \rho_N(dx) =\int\!\! \!\int_{(\mathcal{W}^{-s',q})^2} |u+z|_{\mathcal{W}^{-s',q}}^2 \nu_N(du,dz) \\
& & \le 2\E \big( |u_N(t)|_{\mathcal{W}^{-s',q}}^2 +|Z_N(t)|_{\mathcal{W}^{-s',q}}^2 \big),
\end{eqnarray*}
and the right hand side above is bounded indepently of $N$ and $t$, since $\mu_N$ converges to a Gaussian measure $\mu$ on ${\mathcal{W}^{-s',q}}$. 
The tightness of $(\tilde \rho_N)_N$ follows from Markov inequality
and the compact embedding ${\mathcal{W}^{-s',q}} \subset {\mathcal{W}^{-s,q}}$, for any $s>s'$.
\hfill
\qed

\medskip

 \noindent
 {\em Proof of Theorem \ref{thm:invmeasure}}. Assume now that $\gamma_1$, $\gamma_2$, $s$ and $q$ satisfy the hypothesis of Theorem \ref{thm:invmeasure},
 so that the solution $X$ of equation \eqref{eq:SGL} with $X_0 \in \mathcal{W}^{-s,q}(\R^2)$ is globally defined.
 The tightness of $\tilde{\rho}_N$ ensures in particular that there exists a subsequence weakly converging to 
a measure $\rho$ on $\mathcal{W}^{-s,q}$. Let $P_t$ be the transition semigroup associated with equation \eqref{eq:SGL}
defined by $P_t \phi(X_0)= \E(\phi(X(t,X_0)))$. Since $P_{t}^N \phi(S_N X_0)$ converges to  $P_t \phi(X_0)$ for any $X_0 \in \mathcal{W}^{-s,q}(\R^2)$,
by Proposition  \ref{prop:convGalerkin}, it is not difficult to prove that the limit measure $\rho$ is an invariant measure for $P_t$. 
\hfill\qed

\section{Existence of a stationary martingale solution for any dissipation} 

The aim of this section is to construct a stationary solution of (\ref{eq:SGL}) for any values of $\gamma_1>0$ and $\gamma_2 \in \R$. 
We have seen in the previous section that the system (\ref{eq:coupled}) has a stationary solution $(u_N, Z_N)$ where 
$Z_N=\Pi_N Z^{\gamma_1, \gamma_2}_{\infty}$. Moreover, it is clear that $(u_N, Z^{\gamma_1, \gamma_2}_{\infty})$ is then a stationary solution of   
\begin{equation}
\label{eq:stat}
\left\{ \begin{array}{rcl} \ds
\frac{du}{dt}& =&(\gamma_1+i\gamma_2)\left[ Hu-S_N\big(:|S_N(u+Z)|^2 S_N (u+Z):\big)\right] \\[0.25cm]
dZ &= & (\gamma_1+i\gamma_2) HZ dt +\sqrt{2\gamma_1} dW.
\end{array} \right.
\end{equation}
In this section we will denote $Z^{\gamma_1, \gamma_2}_{\infty}$ by $Z$ for the sake of simplicity. 
Using Proposition \ref{prop:tightness}, we first prove that the the law of this sequence $\{(u_N, Z)\}_{N \in \N}$ is tight 
in an appropriate space to construct a martingale solution.  
\vspace{3mm}

\begin{lem} \label{lem:uz_tendu}Let $\gamma_1>0$, $\gamma_2 \in \R$, $0<s<1$, and $q>8$ such that $qs>8$. Let also $p> \max \{q, 96\}$,  and 
$0<\delta<\frac16$. The sequence $(u_N, Z)_{N\in \N}$ is bounded in 
$$ L^{2m}(\Omega, L^{2m}(0,T, H^{\frac{1}{m}})) \cap L^{\frac{4}{3}}(\Omega, \mathcal{W}^{1,\frac{4}{3}}(0,T, \mathcal{W}^{-2,p})) \times C^{\alpha}([0,T], \mathcal{W}^{-s,q}\cap \mathcal{W}^{-\delta,p})$$
for any $m > 0$, and $\alpha>0$ satisfying $\alpha<\min(\frac{s}{2}-\frac{1}{q}, \frac{\delta}{2}-\frac1p)$. 
\end{lem}

\begin{proof} 
It suffices to check the bound in $L^{\frac{4}{3}}(\Omega, \mathcal{W}^{1,\frac{4}{3}}(0,T, \mathcal{W}^{-2,p}))$, since the
other bounds follow from Proposition \ref{prop:tightness}, the stationarity of $u_N$, and Lemma \ref{lem:Kolmogorov} for $Z$. 
We write the equation for $u_N$; 
\begin{eqnarray*}
{u}_N(t)  
&=& {u}_N(0) + (\gamma_1+i\gamma_2) \int_0^t H {u}_N (\sigma) d\sigma \\ 
&& -(\gamma_1+i\gamma_2) 
\int_0^t S_N(:|S_N({u}_N+ {Z}_N)|^2 S_N ({u}_N+{Z}_N):) (\sigma) d\sigma.   
\end{eqnarray*}
In the right hand side, the first term is constant and clearly bounded in $L^{\frac{4}{3}}(\Omega, \mathcal{W}^{1,\frac{4}{3}}(0,T, \mathcal{W}^{-2,p}))$ 
by Proposition \ref{prop:tightness} and H\"{o}lder inequality. For the second term, we have 
\begin{eqnarray*}
\E \left(\left|\int_0^{\cdot} H u_N(\sigma) d\sigma \right|^{\frac{4}{3}}_{\mathcal{W}^{1,\frac{4}{3}}(0,T, \mathcal{W}^{-2,p})}\right) 
&\le& C_T \E\left(|H u_N|^{\frac{4}{3}}_{L^{\frac{4}{3}}(0,T, \mathcal{W}^{-2,p})} \right) \\
&\le & C_T \E \left(|u_N|^2_{L^2(0,T, \mathcal{W}^{1,2})} \right),
\end{eqnarray*}
where we have used Sobolev embedding and H\"{o}lder inequality in the last inequality. 
To estimate the nonlinear terms, we decompose as in (\ref{def:F}), 
\begin{eqnarray*}
S_N(:|S_N (u_N+Z)|^2 S_N (u_N+Z):)&=&S_N (F_0(S_N u_N)+ F_1(S_N u_N, S_N Z) \\
&& \hspace{3mm} +F_2(S_N u_N, S_N Z)+F_3(S_N Z)).
\end{eqnarray*}
The terms in $F_3$ are simply estimated thanks to Proposition \ref{prop:ReImCauchy} as follows.
$$
\E \left( \left|\int_0^{\cdot} S_N F_3 (s) ds \right|^{\frac{4}{3}}_{\mathcal{W}^{1,\frac{4}{3}}(0,T, \mathcal{W}^{-2,p})}\right) 
\le C_T \E (|S_N F_3|_{L^{\frac{4}{3}}(0,T, \mathcal{W}^{-s,p})}^{\frac{4}{3}}) 
\le C_T M_{s,p}^{\frac{4}{3}}. 
$$
For the term $F_0$, using Sobolev embeddings $L^{\frac43} \subset \mathcal{W}^{-2,p}$ and $\mathcal{W}^{\frac12, 2} \subset L^4$, we obtain  
\begin{eqnarray*}
&& \E (|S_N F_0|_{L^{\frac{4}{3}}(0,T, \mathcal{W}^{-2,p})}^{\frac{4}{3}}) 
= \int_0^T \E(||S_N u_N|^2 S_N u_N|^{\frac{4}{3}}_{\mathcal{W}^{-2,p}}) ds \\
& & \le  C_T \E(||S_N u_N|^2 S_N u_N |^{\frac{4}{3}}_{L^{\frac{4}{3}}}) = C_T \E(| S_N u_N|_{L^4}^4) \le C_T \E(|u_N|^4_{\mathcal{W}^{\frac12, 2}}),
\end{eqnarray*}
which is bounded independently of $N$ by Proposition \ref{prop:tightness}.
To estimate the $F_1$-terms, we fix $s'>0$ such that $s'<\frac{1}{12}$ and $s'p>8$, and apply Lemma \ref{lem:bilinear} to get
\begin{eqnarray*}
\E\left(\left|S_N F_1\right|_{L^{\frac{4}{3}}(0,T, \mathcal{W}^{-2,p})}^{\frac{4}{3}}\right) 
\le C_T \E \left(|F_1|_{\mathcal{W}^{-(s'+\frac{14}{24}), p}}^{\frac{4}{3}} \right)
&\le& C_T \E \left(|S_N Z|_{\mathcal{W}^{-s',p}}^{\frac{4}{3}} |S_N u_N|_{\mathcal{W}^{\frac{3}{8},3}}^{\frac{8}{3}} \right).
\end{eqnarray*}
Using then the Sobolev embedding $\mathcal{W}^{\frac{17}{24},2} \subset \mathcal{W}^{\frac{3}{8},3}$ and H\"older inequality, 
the right hand side is majorized by 
\begin{eqnarray*}
C_T \E(|S_N Z|_{\mathcal{W}^{-s',p}}^{24})^{\frac{1}{18}}~ \E(|S_N u_N|^{\frac{48}{17}}_{\mathcal{W}^{\frac{17}{24},2}})^{\frac{17}{18}},  
\end{eqnarray*}
and is thus bounded independently of $N$ thanks to Propositions \ref{prop:ReImCauchy} and \ref{prop:tightness}. 
Finally, for the terms in $F_2$, we apply again Lemma \ref{lem:bilinear}: 
\begin{eqnarray*}
\E\left(\left|S_N F_2\right|_{L^{\frac{4}{3}}(0,T, \mathcal{W}^{-2,p})}^{\frac{4}{3}}\right) 
&\le& T\E \left(|F_2|_{\mathcal{W}^{-(s'+\frac{1}{6}), p}}^{\frac{4}{3}} \right)\\
&\le& C_T \E(\sum_{k+l=2}|:(S_N Z_R)^k (S_N Z_I)^l :|_{\mathcal{W}^{-s',p}}^{\frac{4}{3}} |S_N u_N|_{\mathcal{W}^{\frac13,4}}^{\frac43})\\
&\le & C_T \sum_{k+l=2} \E(|:(S_N Z_R)^k (S_N Z_I)^l :|_{\mathcal{W}^{-s',p}}^{3})^{\frac49}~ \E(|u_N|_{\mathcal{W}^{\frac56,2}}^{\frac{12}{5}})^{\frac59}, 
\end{eqnarray*}
which is bounded again by Propositions \ref{prop:ReImCauchy} and \ref{prop:tightness}. Note that  
in the last inequality we have used the Sobolev embedding $\mathcal{W}^{\frac56, 2} \subset \mathcal{W}^{\frac13,4}$ and H\"older inequality. 
\end{proof}

\begin{rem} \label{rem:compactness}
We note that by Lemma \ref{lem:uz_tendu}, $(u_N)_{N\in \N}$ is bounded in 
\begin{eqnarray*}
&& L^3(\Omega, L^3(0,T, \mathcal{W}^{\frac23,2})) \cap L^{\frac43}(\Omega, \mathcal{W}^{1,\frac43}(0,T, \mathcal{W}^{-2,p})) \\
&& \hspace{3cm} \subset 
L^{\frac43}(\Omega, L^3(0,T, \mathcal{W}^{\frac23,2})) \cap L^{\frac43}(\Omega, \mathcal{W}^{\frac{1}{12},3}(0,T, \mathcal{W}^{-2,p})), 
\end{eqnarray*}
and $L^3(0,T, \mathcal{W}^{\frac23,2}) \cap \mathcal{W}^{\frac{1}{12},3}(0,T, \mathcal{W}^{-2,p})$ is compactly embedded in $L^3(0,T, L^4_x)$.
On the other hand, $(u_N)_{N\in \N}$ is also bounded in 
\begin{eqnarray*}
&& L^2(\Omega, L^2(0,T, \mathcal{W}^{1,2})) \cap L^{\frac43}(\Omega, \mathcal{W}^{1,\frac43}(0,T, \mathcal{W}^{-2,p})) \\ 
&& \hspace{3cm} \subset L^{\frac43}(\Omega, L^2(0,T, \mathcal{W}^{1,2}) \cap \mathcal{W}^{\frac{1}{12},2}(0,T,\mathcal{W}^{-2,p})), 
\end{eqnarray*}
and $L^2(0,T, \mathcal{W}^{1,2}) \cap \mathcal{W}^{\frac{1}{12},2}(0,T,\mathcal{W}^{-2,p})$ is compactly embedded
in $L^2(0,T,\mathcal{W}^{s,2})$ for any $s$ with $0 \le s <1$. In particular, since  
$\mathcal{W}^{\frac56, 2} \subset \mathcal{W}^{\frac13,4}$, the embedding is compact in  
$L^2(0,T, \mathcal{W}^{\frac13,4}).$ 
Finally, we note that $\mathcal{W}^{1,\frac43}(0,T,\mathcal{W}^{-2,p})$ is compactly embedded in 
$C([0,T], \mathcal{W}^{-3,p})$. 
\end{rem}

\begin{proof}[Proof of Theorem \ref{thm:sol_martingale}] Let $\alpha>0$ satisfy the condition in Lemma \ref{lem:uz_tendu}.  
We deduce from Lemma \ref{lem:uz_tendu}, Remark \ref{rem:compactness} and Markov inequality that the sequence $\{(u_N, Z)_{N\in \N}\}$ is tight in 
\begin{equation} \label{space:compactness}
L^3(0,T, L^4) \cap L^2(0,T, \mathcal{W}^{\frac56,2}) \cap C([0,T], \mathcal{W}^{-3,p}) \times C^{\beta}([0,T], \mathcal{W}^{-s', q}\cap \mathcal{W}^{-\delta,p})
\end{equation}
for any $\beta <\alpha$ and $s'> s$. Fix such $\beta$ and $s'$. By Prokhorov Theorem, there exists a subsequence, still denoted $\{(u_N, Z)_{N\in \N}\}$ which converges 
in law to a measure $\nu$ on the space \eqref{space:compactness}.
By Skorokhod Theorem, there exist $(\tilde{\Omega}, \tilde{\mathcal{F}}, \tilde{\prob})$, $(\tilde{u}_N, \tilde{Z}_N)_{N \in \N}$ and $(\tilde{u}, \tilde{Z})$ 
taking values in the same space \eqref{space:compactness}, satisfying $\mathcal{L}((u_N, Z))=\mathcal{L}((\tilde{u}_N, \tilde{Z}_N))$ for any $N \in \N$, $\mathcal{L}((\tilde{u}, \tilde{Z}))=\nu$, 
and $\tilde{u}_N$ converges to  $\tilde{u}$, $\tilde{\prob}$-a.s. in $L^3(0,T, L^4) \cap L^2(0,T, \mathcal{W}^{\frac56,2}) \cap C([0,T], \mathcal{W}^{-3,p})$, 
$\tilde{Z}_N$  converges to $\tilde{Z}$, $\tilde{\prob}$-a.s. in $C^{\beta}([0,T], \mathcal{W}^{-s', q}\cap \mathcal{W}^{-\delta,p})$. 
Moreover, by diagonal extraction, it can be assumed that this holds for any $T>0$. It is easily seen that $(\tilde{u}, \tilde{Z})$ is a stationary process thanks to 
the convergence of $(\tilde{u}_N, \tilde{Z}_N)$ to $(\tilde{u}, \tilde{Z})$ in $C([0,T], \mathcal{W}^{-3,p}) \times C^{\beta}([0,T], \mathcal{W}^{-s', q})$. This convergence also implies
$\mathcal{L}(Z)=\mathcal{L}(\tilde{Z})$. Note also that if we extract from the subsequence used in the proof of Theorem \ref{thm:invmeasure}, we have that 
for each $t\in \R$, $\mathcal{L}(\tilde X(t) )=\mathcal{L}(\tilde u (t)+\tilde Z(t))=\rho$.

Write then, 
\begin{eqnarray} \nonumber
\tilde{u}_N(t)- \tilde{u}_N(0) 
&=& (\gamma_1+i\gamma_2) \int_0^t H \tilde{u}_N (\sigma) d\sigma \\ \label{eq:uz_N}
&& -(\gamma_1+i\gamma_2) 
\int_0^t S_N(:|S_N(\tilde{u}_N+ \tilde{Z}_N)|^2 S_N (\tilde{u}_N+\tilde{Z}_N):) (\sigma) d\sigma.   
\end{eqnarray}
It remains us to show that the right hand side of \eqref{eq:uz_N} converges, up to a subsequence, to 
\begin{eqnarray*} 
(\gamma_1+i\gamma_2) \int_0^t H \tilde{u} (\sigma) d\sigma  
-(\gamma_1+i\gamma_2) 
\int_0^t :|(\tilde{u}+ \tilde{Z})|^2 (\tilde{u}+\tilde{Z}): (\sigma) d\sigma,  
\end{eqnarray*}
$\tilde{\prob}$-a.s. in $C([0,T], \mathcal{W}^{-2,p})$. This can be checked as follows. First, the convergence of 
the linear term follows from the convergence of $\tilde{u}_N$ to $\tilde{u}$ 
in $L^2(0,T, \mathcal{W}^{\frac56,2}) \subset L^1(0,T, \mathcal{W}^{-2,p})$. In order to prove the convergence 
of nonlinear terms, we again decompose the nonlinear terms into the four terms $F_0, \dots,F_3$ as in (\ref{def:F}) and estimate them separately, i.e.,
\begin{eqnarray*}
&& \sup_{t\in [0,T]} \left|\int_0^t :|\tilde{u}+\tilde{Z}|^2 (\tilde{u}+\tilde{Z}): (\sigma)d\sigma 
-\int_0^t S_N(:|S_N(\tilde{u}_N+ \tilde{Z}_N)|^2 S_N (\tilde{u}_N+\tilde{Z}_N):) (\sigma) d\sigma  \right|_{\mathcal{W}^{-2,p}}\\ 
& \le & \sum_{k=0}^3 \int_0^T |F_k(\tilde{u}, \tilde{Z})-S_N F_k(S_N \tilde{u}_N, S_N \tilde{Z}_N)|_{\mathcal{W}^{-2,p}} d\sigma \\
&=& I_0^N+I_1^N+I_2^N+I_3^N.
\end{eqnarray*}
We begin with the convergence of $I_0^N$. 
\begin{eqnarray*}
I_0^N &=& \int_0^T |S_N(|S_N \tilde{u}_N|^2 S_N \tilde{u}_N) -|\tilde{u}|^2 \tilde{u}|_{\mathcal{W}^{-2,p}} d\sigma\\
&\le& \int_0^T ||S_N \tilde{u}_N|^2 S_N \tilde{u}_N - |\tilde{u}|^2 \tilde{u}|_{L^{\frac43}} d\sigma 
+ \int_0^T |(S_N -I) |\tilde{u}|^2 \tilde{u}|_{L^{\frac43}} d\sigma. 
\end{eqnarray*}
Since $\tilde{u}_N $ converges to $\tilde{u}$ a.s. in $L^3(0,T, L^4_x)$, by dominated convergence, the same holds for $S_N \tilde{u}_N$,
thus $|S_N \tilde{u}_N|^2 S_N \tilde{u}_N$ converges to $|\tilde{u}|^2 \tilde{u}$ a.s. in $L^{\frac{4}{3}}(0,T, L^{\frac43})$. 
Similarly, the second term converges to zero by dominated convergence, 
therefore, $I_0^N$ converges to $0$. Next, we use Lemma \ref{lem:bilinear}
to obtain
\begin{eqnarray*}
I_1^N &\le& \int_0^T |F_1(\tilde{u}, \tilde{Z})-S_N F_1(S_N \tilde{u}_N, S_N \tilde{Z}_N)|_{\mathcal{W}^{-(\delta+\frac13),p}} d\sigma \\
&\lesssim& |\tilde{Z} - S_N \tilde{Z}_N|_{C([0,T],\mathcal{W}^{-\delta,p})} \left(|S_N \tilde{u}_N|^2_{L^2(0,T, \mathcal{W}^{\frac13,4})} + 
|\tilde{u}|^2_{L^2(0,T, \mathcal{W}^{\frac13,4})}\right) \\
&& \hspace{3mm}+  \left(|S_N \tilde{Z}_N|_{C([0,T], \mathcal{W}^{-\delta,p})} + 
|\tilde{Z}|_{C([0,T], \mathcal{W}^{-\delta,p})}\right)
\left(|S_N \tilde{u}_N|_{L^2(0,T, \mathcal{W}^{\frac13,4})} + 
|\tilde{u}|_{L^2(0,T, \mathcal{W}^{\frac13,4})}\right) \\
&& \hspace{5cm} \times |S_N \tilde{u}_N - \tilde{u}|_{L^2(0,T, \mathcal{W}^{\frac13,4})}.
\end{eqnarray*}
Hence, $I_1^N$ converges to $0$  since $S_N \tilde{Z}_N$ converges to $\tilde{Z}$ in $C([0,T], \mathcal{W}^{-\delta,p})$ and, by Remark \ref{rem:compactness},
$S_N \tilde{u}_N$ converges to $\tilde{u}$ in $L^2(0,T, \mathcal{W}^{\frac13,4})$. 
Concerning $F_2$, we proceed as for $F_1$, and we use Lemma \ref{lem:bilinear} in the same way;
\begin{eqnarray*}
I_2^N &\lesssim& \sum_{k+l=2} (|:\tilde{Z}_R^k \tilde{Z}_I^l :|_{L^2(0,T, \mathcal{W}^{-\delta,p})} +|S_N(:(S_N \tilde{Z}_{N,R})^k (S_N \tilde{Z}_{N,I})^l :)|_{L^2(0,T, \mathcal{W}^{-\delta,p})} ) \\
&& \times |S_N \tilde{u}_N -\tilde{u}|_{L^2(0,T, \mathcal{W}^{\frac13,4})} \\
&& + \sum_{k+l=2}  |:\tilde{Z}_R^k \tilde{Z}_I^l :- S_N(: (S_N \tilde{Z}_{N,R})^k (S_N \tilde{Z}_{N,I})^l :) |_{L^2(0,T, \mathcal{W}^{-\delta,p})} \\
&& \times \left(|S_N \tilde{u}_N|_{L^2(0,T, \mathcal{W}^{\frac13,4})} + |\tilde{u}|_{L^2(0,T, \mathcal{W}^{\frac13,4})} \right).
\end{eqnarray*}
Note that Proposition \ref{prop:ReImCauchy} implies the convergence to zero of the second term in $L^p(\Omega)$, thus extracting a subsequence, 
$\tilde{\prob}$-a.s.  The first term goes to $0$, too since, again, $S_N \tilde{u}_N$ converges to $\tilde{u}$ in $L^2(0,T, \mathcal{W}^{\frac13,4})$. 
The term $I_3^N$ can be treated similarly. 
We deduce from this convergence result that $\tilde u$ satisfies
$$
\tilde{u}(t)- \tilde{u}(0) 
= (\gamma_1+i\gamma_2) \int_0^t H \tilde{u} (\sigma) d\sigma  -(\gamma_1+i\gamma_2) 
\int_0^t :|(\tilde{u}+ \tilde{Z})|^2 (\tilde{u}+\tilde{Z}):) (\sigma) d\sigma.   
$$
Since it is clear that $\tilde Z$ satisfies the second equation in \eqref{eq:stat}, we easily deduce that $\tilde X=\tilde u +\tilde Z$ is a stationary solution of
\eqref{eq:SGL}  on $(\tilde \Omega, \tilde F, \tilde \prob)$. Moreover, it is not difficult to prove that $\tilde u$ is continuous with values in 
$\mathcal{W}^{-s,q}$, which ends the proof of Theorem \ref{thm:sol_martingale}.
\end{proof}

\section{Appendix 1} 

\noindent
{\em Proof for Lemma \ref{lem:regularization}}.
It suffices to show (\ref{ineq:regularization}) for the case $s=2$. Indeed, 
all other cases $s\in [0,2)$ follow by the interpolation 
\begin{equation*}
|f|_{\mathcal{W}^{s,p}} \le |f|_{\mathcal{W}^{s_0,p}}^{1-\theta} |f|_{\mathcal{W}^{s_1,p}}^{\theta} 
\end{equation*}
with $s=(1-\theta)s_0 + \theta s_1$, and $0<\theta<1$, 
since Lemma \ref{lem:heat_kernel} gives the inequality in the case $s=0$.
Due to the Mehler formula, the expression of the kernel of $e^{t(\gamma_1+i\gamma_2)H}$ for $d=2$ is given by
\begin{equation*}
e^{t(\gamma_1+i\gamma_2)H}f(x)= \int_{\R^2} {\Xi}_t(x,y) f(y) dy, \quad t>0, \quad x\in \R^2,
\end{equation*}
with  
\begin{eqnarray*}
{\Xi}_t (x,y)&:=&\frac{\exp \Big[ -\frac{\mathrm{cosh}(2t(\gamma_1+i\gamma_2))}{2\mathrm{sinh} (2t(\gamma_1+i\gamma_2))} (|x|^2+|y|^2)
+\frac{1}{\mathrm{sinh} (2t(\gamma_1+i\gamma_2))} x \cdot y \Big] }{2\pi \mathrm{sinh} (2t(\gamma_1+i\gamma_2))} 
\\
&=& \frac{\delta}{\pi} \exp({-(\beta-\delta)|x|^2}) \, \exp({-\delta |x-y|^2}) \, \exp({-(\beta-\delta)|y|^2}),
\end{eqnarray*}
where we have set
\begin{eqnarray*}
\delta = \frac{1}{2 \mathrm{sinh} (2t(\gamma_1+i\gamma_2))}, \quad 
\beta = \frac{\mathrm{cosh}(2t(\gamma_1+i\gamma_2))}{2\mathrm{sinh} (2t(\gamma_1+i\gamma_2))}.
\end{eqnarray*}
Note that $\re (\delta) > 0$ for $0< t < \frac{\pi}{4|\gamma_2|}$ if $\gamma_2 \ne 0$, and for all $t>0$ if $\gamma_2=0$. Also, 
$\mathrm{Re}(\beta-\delta) > 0$ for any $t>0$.
Here we rewrite the kernel of $e^{t(\gamma_1+i\gamma_2)H}$ as 
$$ (e^{(\gamma_1+i\gamma_2)tH}f) (x)=e^{-(\beta-\delta)|x|^2} (e^{\delta^{-1} \Delta} g)(x), \quad g(x)=e^{-(\beta-\delta)|x|^2} f(x),$$ 
where we denoted $e^{\delta^{-1} \Delta}$ the evolution operator associated to the kernel $ \frac{\delta}{\pi} e^{-\delta|x-y|^2}$; thus 
$$ (e^{\delta^{-1} \Delta} g)(x)= \int_{\R^2_y} \frac{\delta}{\pi} e^{-\delta|x-y|^2} g(y) dy= (U \ast g)(x)$$
with the notation 
$$ U(x)=\frac{\delta}{\pi} e^{-\delta|x|^2}.$$ 
Remark that $\mathcal{F}(U)(\xi)= e^{-\frac{|\xi|^2}{4\delta}}$ since $\mathrm{Re} (\delta) >0,$ where $\mathcal{F}$ denotes the Fourier transform.
\vspace{3mm}

Thanks to the norm equivalence in Proposition \ref{prop:dg} (1), we shall estimate for any $1< p<\infty$,
$|\langle D \rangle^{2} e^{(\gamma_1+i\gamma_2)tH} f|_{L^p}$ and $|\langle x \rangle^{2} e^{(\gamma_1+i\gamma_2)tH} f|_{L^p}$,
and show that these terms are bounded by $Ct^{-1}$.

First,
\begin{eqnarray} \nonumber 
\langle D \rangle^2 e^{(\gamma_1+i\gamma_2)tH} f(x) &=& \mathcal{F}^{-1} \mathcal{F}\{\langle D \rangle^2e^{-(\beta-\delta)|x|^2} (e^{\delta^{-1} \Delta} g)\} \\ \nonumber
                                   &=& \frac{\pi}{\beta-\delta} \mathcal{F}^{-1} [\langle \xi\rangle^2 (e^{-\frac{|\cdot|^2}{4(\beta-\delta)}} \ast \mathcal{F} (e^{\delta^{-1} \Delta} g))]\\ \nonumber 
                                   &=& \frac{\pi}{\beta-\delta} \int_{\R^2_{\xi}} e^{ix\cdot \xi} \langle\xi\rangle^2 \Big[ \int_{\R^2_{\eta}} 
                                   e^{-\frac{|\xi-\eta|^2}{4(\beta-\delta)}} \widehat{U \ast g}(\eta) d\eta\Big] d\xi\\ \label{ineq:conv}
                                   &=& (\mathrm{i})+(\mathrm{ii})+(\mathrm{iii})+(\mathrm{iv}), 
\end{eqnarray}
where, using the Fubini theorem, the change of variables $\xi'=\xi-\eta$, and developing the bracket 
$\langle \xi'+\eta \rangle^2 = \langle \xi' \rangle^2 + \langle \eta \rangle^2 -1 + 2 \xi' \cdot \eta$ in the last equality, 
\begin{eqnarray*}
(\mathrm{i})&=& -\frac{\pi}{\beta-\delta} \int_{\R^2_{\xi'}} \int_{\R^2_{\eta}} e^{ix \cdot (\xi'+\eta)} e^{-\frac{|\xi'|^2}{4(\beta-\delta)}}  \widehat{U \ast g}(\eta) d\eta d\xi' \\
(\mathrm{ii}) &=& \frac{\pi}{\beta-\delta} \int_{\R^2_{\xi'}} \int_{\R^2_{\eta}} e^{ix \cdot (\xi'+\eta)} \langle \xi' \rangle^2 e^{-\frac{|\xi'|^2}{4(\beta-\delta)}}  \widehat{U \ast g}(\eta) d\eta d\xi' \\
(\mathrm{iii}) &=& \frac{\pi}{\beta-\delta} \int_{\R^2_{\xi'}} \int_{\R^2_{\eta}} e^{ix \cdot (\xi'+\eta)} \langle \eta \rangle^2 e^{-\frac{|\xi'|^2}{4(\beta-\delta)}}  \widehat{U \ast g}(\eta) d\eta d\xi' \\
(\mathrm{iv}) &=& \frac{2\pi}{\beta-\delta} \int_{\R^2_{\xi'}} \int_{\R^2_{\eta}} e^{ix \cdot (\xi'+\eta)} \xi' \cdot \eta e^{-\frac{|\xi'|^2}{4(\beta-\delta)}}  \widehat{U \ast g}(\eta) d\eta d\xi'.  
\end{eqnarray*}
Note that (i) may be written as $-e^{-(\beta-\delta)|x|^2} (U \ast g)(x)$, thus 
$$|(\mathrm{i})|_{L^p} \le |U \ast g|_{L^p} \le |U|_{L^1}|g|_{L^p} \le |f|_{L^p},$$
by Young inequality. On the other hand, (ii) equals $(\langle D \rangle^2 e^{-(\beta-\delta)|x|^2}) (U \ast g)(x)$, therefore we have
$$|(\mathrm{ii})|_{L^p} \le |\langle D \rangle^2 e^{-(\beta-\delta)|x|^2}|_{L^{\infty}}|U \ast g|_{L^p} \le C_{\gamma_1,\gamma_2}(1+|\beta-\delta |)  |f|_{L^p},$$
for $t>0$ with $\gamma_1 t, |\gamma_2 t| <<1$, since 
\begin{eqnarray*}
|\langle D \rangle^2 e^{-(\beta-\delta)|x|^2}|_{L^{\infty}} &\le&  C_{\gamma_1, \gamma_2} (1+|\beta-\delta |)
\end{eqnarray*}
for $t>0$ with $\gamma_1 t, |\gamma_2 t| <<1$. Here we used the fact that 
$|\frac{\mathrm{Im}(\beta-\delta)}{\mathrm{Re}(\beta-\delta)}| \le  C_{\gamma_1, \gamma_2} $ for $t>0$ with $\gamma_1 t, |\gamma_2 t| <<1$.
Next, $(\mathrm{iii})=e^{-(\beta-\delta)|x|^2} \langle D \rangle^2 (U\ast g)(x)$, thus we obtain
$$|(\mathrm{iii})|_{L^p} \le |(\langle D \rangle^2 U) \ast g |_{L^p} \le C(1+|\delta|)  |f|_{L^p}.$$
Similarly as above, we estimate 
\begin{eqnarray*}
|(\mathrm{iv})|_{L^p} &\le& C|\langle D \rangle e^{-(\beta-\delta)|x|^2}|_{L^{\infty}} |\langle D \rangle (U \ast g)|_{L^p} \\
&\le &  C_{\gamma_1, \gamma_2}  (1+|\beta-\delta|)^{\frac 12} (1+|\delta |)^{\frac 12} |f|_{L^p}
\end{eqnarray*}
for $t>0$ with $\gamma_1 t, |\gamma_2 t| <<1$. 

The other term may be estimated as follows using the notation
$$|f|_{L^p_\sigma}^p= \int_{\R^2} |f(x)|^p \langle x \rangle^\sigma dx, \quad \sigma \ge 0:$$
\begin{eqnarray*}
|\langle x \rangle^{2} e^{(\gamma_1+i\gamma_2)tH} f|_{L^p} &=&  |e^{(\gamma_1+i\gamma_2)tH} f|_{L^p_{2p}} = |e^{-(\beta-\delta)|x|^2} (e^{\delta^{-1} \Delta} g)|_{L^p_{2p}} \\
           & \le & \sup_{x\in \R^2} (e^{-\mathrm{Re}(\beta-\delta)|x|^2} \langle x \rangle^2) |e^{\delta^{-1} \Delta} g|_{L^p} \\
                                                     &\le&  \frac{C}{\mathrm{Re}(\beta-\delta)} |f|_{L^p}. 
\end{eqnarray*}
The proof is completed by the fact that $\re(\beta-\delta) \sim \frac{\gamma_1 }{2}t$ and $|\delta | \sim C_{\gamma_1,\gamma_2} t^{-1}$, for 
$t>0$ with $\gamma_1 t, |\gamma_2 t| <<1$.
\hfill
\qed
\vspace{3mm}

\noindent
{\em Proof of Lemma \ref{lem:arnaud}}. 
Let $\delta>0$; by H\"{o}lder inequality, for any $g \in \mathcal{W}^{\delta, 2p}$ 
with $\delta p^* > 1$ and $\frac{1}{p}+\frac{1}{p^*}=1,$
$$|g|_{L^2}^2=\int_{\R^2}(1+|x|^2)^{-\delta} (1+|x|^2)^{\delta} |g(x)|^2 dx \le C \left[\int_{\R^2} (1+|x|^2)^{\delta p} |g(x)|^{2p} dx\right]^{\frac{1}{p}}.$$
By (1) of Proposition \ref{prop:dg}, 
$$ \left[\int_{\R^2} (1+|x|^2)^{\delta p} |g(x)|^{2p} dx\right]^{\frac{1}{2p}} \le C |(-{H})^{\frac{\delta}{2}} g|_{L^{2p}}.$$
Writing $f=(-{H})^{\frac{\delta}{2}} g$, we then deduce
$|(-{H})^{-\frac{\delta}{2}} f|_{L^2} \le C|f|_{L^{2p}},$ if $\delta p^* > 1$, i.e. if $\delta > 1-\frac{1}{p}.$ 
Then we get 
\begin{eqnarray*}
|e^{t(\gamma_1+i\gamma_2){H}}f|_{L^2}
&=&|(-{H})^{\frac{\delta}{2}} e^{t(\gamma_1+i\gamma_2){H}}(-{H})^{-\frac{\delta}{2}}f|_{L^2} 
\le C|e^{t(\gamma_1+i\gamma_2){H}}(-{H})^{-\frac{\delta}{2}}f|_{\mathcal{W}^{\delta,2}} \\
&\le & C  t^{-\frac{\delta}{2}}|(-{H})^{-\frac{\delta}{2}}f|_{L^2} 
\le C  t^{-\frac{\delta}{2}}|f|_{L^{2p}},
\end{eqnarray*}
where we have used Lemma \ref{lem:regularization} in the fourth inequality.
Replacing $2p$ by $p$ and writing $\beta=\frac{\delta}{2}$, we have for $p\ge 2$, $t>0$,
$$|e^{t(\gamma_1+i\gamma_2){H}} f|_{L^2} \le C  t^{-\beta}|f|_{L^p}$$ 
for $\beta>\frac{1}{2}-\frac{1}{p}$, which gives the expected result.
\hfill
\qed

\section{Appendix 2} 
In this section we give a proof of Lemma \ref{lem:bilinear}. Before proving Lemma \ref{lem:bilinear} we need 
two preliminary results.
\vspace{3mm}

\begin{lem} \label{lem:bilinear_positive}
Let $\beta, \gamma \ge \alpha >0$, $\gamma <\frac{2}{p_2}$, $\beta<\frac{2}{p_1}$, $1<q,p_1, p_2<+\infty$, 
and 
$$ \frac{1}{p_1} +\frac{1}{p_2} - \frac{\beta}{2}-\frac{\gamma}{2}+\frac{\alpha}{2}=\frac{1}{q}.$$
Then, for any $f \in \mathcal{W}^{\beta,p_1}(\R^2)$, $g\in \mathcal{W}^{\gamma, p_2}(\R^2)$, 
$$|fg|_{\mathcal{W}^{\alpha, q}} \lesssim |f|_{\mathcal{W}^{\beta,p_1}} |g|_{\mathcal{W}^{\gamma, p_2}}.$$
\end{lem}
\proof Under the condition on the parameters, we may find, by the Sobolev embedding, $q_1$ with $1 < p_1 \le q_1 <+\infty$ such that 
$\frac{1}{q_1}=\frac{1}{p_1}-\frac{\beta}{2}, $ so that $|f|_{L^{q_1}} \lesssim |f|_{\mathcal{W}^{\beta,p_1}}$. On the other hand, 
we have $\frac{2}{p_1} \ge \frac{2}{p_1} -(\beta-\alpha)>0$, thus there exists $q_2$ such that 
$1 <p_1 \le q_2 <+\infty$ and $\frac{1}{q_2}=\frac{1}{p_1}-\frac{\beta-\alpha}{2}$, i.e., 
$|f|_{\mathcal{W}^{\alpha,q_2}} \lesssim |f|_{\mathcal{W}^{\beta,p_1}}$. 
Similarly, we may find $\tilde{q_2}$ with $1<p_2 < \tilde{q_2} <+\infty$ such that 
$\frac{1}{\tilde{q_2}}=\frac{1}{p_2}-\frac{\gamma}{2}$, i.e. $|g|_{L^{\tilde{q_2}}} \lesssim |g|_{\mathcal{W}^{\gamma,p_2}}$, 
and we may find $\tilde{q_1}$ with $1<p_2 \le \tilde{q_1} <+\infty$ such that 
$\frac{1}{\tilde{q_1}}=\frac{1}{p_2}-\frac{\gamma-\alpha}{2}$, i.e. $|g|_{\mathcal{W}^{\alpha,\tilde{q_1}}} \lesssim |g|_{\mathcal{W}^{\gamma,p_2}}$. 
At last, the relation 
$ \frac{1}{p_1} +\frac{1}{p_2} - \frac{\beta}{2}-\frac{\gamma}{2}+\frac{\alpha}{2}=\frac{1}{q}$
implies 
$ \frac{1}{q_1} +\frac{1}{\tilde{q_1}}= \frac{1}{q_2}+\frac{1}{\tilde{q_2}}=\frac{1}{q}.$
Therefore, applying Proposition \ref{prop:dg} (2) and obtain 
$$|fg|_{\mathcal{W}^{\alpha,q}} \le C(|f|_{L^{q_1}} |g|_{\mathcal{W}^{\alpha, \tilde{q_1}}} + |f|_{\mathcal{W}^{\alpha, q_2}} |g|_{L^{\tilde{q_2}}}).$$
Combining with the above inequalities leads to the desired result. 
\hfill\qed
\vspace{3mm}

With the use of a duality argument, we obtain the following lemma. 
\vspace{3mm}

\begin{lem} \label{lem:bilinear_dual}Let $\beta, \gamma \ge \alpha >0$, $\gamma<2\left(1-\frac{1}{q}\right)$, $\beta<\frac{2}{p_2}$, $1<q,p_1, p_2<+\infty$,
and 
$$ \frac{1}{p_1} +\frac{1}{p_2} - \frac{\beta}{2}-\frac{\gamma}{2}+\frac{\alpha}{2}=\frac{1}{q}.$$
Then, for any $f \in \mathcal{W}^{\beta,p_2}(\R^2)$, $h\in \mathcal{W}^{-\alpha, p_1}(\R^2)$, 
$$|fh|_{\mathcal{W}^{-\gamma, q}} \lesssim |f|_{\mathcal{W}^{\beta,p_2}} |h|_{\mathcal{W}^{-\alpha, p_1}}.$$
\end{lem}
\proof By duality for $f\in \mathcal{W}^{\beta, p_2}$, $g \in \mathcal{W}^{\gamma, \tilde{q}}$ and $h\in \mathcal{W}^{-\alpha,p_1}$ 
with $\frac{1}{q}+\frac{1}{\tilde{q}}=1$, $\frac{1}{p_1}+\frac{1}{\tilde{p_1}}=1$, we have 
$$ (hf,g)=(h,fg) \le |h|_{\mathcal{W}^{-\alpha, p_1}} |fg|_{\mathcal{W}^{\alpha,\tilde{p_1}}}.$$
In order to use Lemma \ref{lem:bilinear_positive}, let us check the conditions on the parameters : 
by the assumption, we have $\beta<\frac{2}{p_2}$ and $ \gamma<\frac{2}{\tilde{q}}$. 
On the other hand, 
$\frac{1}{p_2} + \frac{1}{\tilde{q}}-\frac{\beta}{2}-\frac{\gamma}{2}+\frac{\alpha}{2}=\frac{1}{p_2}+1-\frac{1}{q}-\frac{\beta}{2}-\frac{\gamma}{2}+\frac{\alpha}{2}$,
which in turn is equal to $1-\frac{1}{p_1}=\frac{1}{\tilde{p_1}}.$ 
Thus it follows from Lemma \ref{lem:bilinear_positive} that 
$|fg|_{\mathcal{W}^{\alpha, \tilde{p_1}}} \lesssim |f|_{\mathcal{W}^{\beta,p_2}} |g|_{\mathcal{W}^{\gamma, \tilde{q}}},$
which implies
$ |hf|_{\mathcal{W}^{-\gamma, \tilde{q}}} \lesssim |h|_{\mathcal{W}^{-\alpha, p_1}} |f|_{\mathcal{W}^{\beta,p_2}}.$
\hfill\qed
\vspace{3mm} 

\noindent
{\em Proof of Lemma \ref{lem:bilinear}}. It suffices to apply Lemma \ref{lem:bilinear_dual} with 
$$ \gamma=s+l\left(\frac{2}{p}-\beta \right) \ge s+ (l-1)\left(\frac{2}{p}-\beta \right) =\alpha, \quad p_1=q, \quad p_2=p$$
to obtain 
\begin{equation} \label{eq:one_step}
|hf^l|_{\mathcal{W}^{-\left(s+l(\frac{2}{p}-\beta )\right),q}} 
\le C |hf^{l-1}|_{\mathcal{W}^{-\left(s+(l-1)(\frac{2}{p}-\beta )\right),q}} |f|_{\mathcal{W}^{\beta,p}}.
\end{equation}
Repeating recursively the argument gives the result. 
\hfill\qed
\vspace{3mm}

{\bf Acknowledgements.} 
This work was supported by JSPS KAKENHI Grant Numbers JP19KK0066, JP20K03669.
A. Debussche is partially supported by the French government thanks to the "Investissements d'Avenir" program ANR-11-LABX-0020-0, Labex Centre Henri Lebesgue.

\end{document}